\documentclass{amsart}
\usepackage{amsmath,amssymb,amscd,latexsym}
\usepackage[all]{xy}

\newcommand{\Z}{{\mathbb{Z}}}

\newcommand{\heta}{\hat{\eta}}

\newcommand{\et}{\mathrm{\acute{e}t}}

\newcommand{\ok}{\bar{k}}

\newcommand{\Xok}{X_{\ok}}
\newcommand{\pro}{\mathrm{pro}}

\newcommand{\Aut}{\mathrm{Aut}}

\newcommand{\colim}{\operatorname*{colim}}
\newcommand{\hocolim}{\operatorname*{hocolim}}

\newcommand{\cosk}{\mathrm{cosk}}

\newcommand{\hEt}{\hat{\mathrm{Et}}\,}

\newcommand{\compl}{\hat{(\cdot)}}

\newcommand{\Gal}{\mathrm{Gal}}

\newcommand{\Hom}{\mathrm{Hom}}

\newcommand{\Ker}{\mathrm{Ker}\,}

\newcommand{\sk}{\mathrm{sk}}

\newcommand{\Ch}{{\mathcal C}}

\newcommand{\Eh}{{\mathcal E}}
\newcommand{\hEh}{\hat{\mathcal E}}
\newcommand{\Fh}{{\mathcal F}}

\newcommand{\Hh}{{\mathcal H}}
\newcommand{\Hhf}{{\mathcal H}_{0,\mathrm{fin}}}

\newcommand{\hHh}{\hat{{\mathcal H}}}

\newcommand{\hHhg}{\hat{{\mathcal H}}_G}

\newcommand{\Mh}{{\mathcal M}}

\newcommand{\Ph}{\mathcal{P}}

\newcommand{\Rh}{{\mathcal R}}
\newcommand{\Sh}{{\mathcal S}}
\newcommand{\Shp}{{\mathcal S}_{\ast}}

\newcommand{\hSh}{\hat{\mathcal S}}
\newcommand{\hShg}{\hat{\mathcal S}_G}

\newcommand{\hShp}{\hat{\mathcal S}_{\ast}}

\newtheorem{theorem}{Theorem}[section]
\newtheorem{lemma}[theorem]{Lemma}
\newtheorem{prop}[theorem]{Proposition}

\newtheorem{cor}[theorem]{Corollary}
\theoremstyle{definition}

\newtheorem{example}[theorem]{Example}

\newtheorem{remark}[theorem]{Remark}

\newtheorem{defn}{Definition}\numberwithin{defn}{section}
\title[Some remarks on profinite completion]
{Some remarks on profinite completion of spaces}
\author[Gereon Quick]{Gereon Quick}
\thanks{Supported by German Research Foundation (DFG) Fellowship QU 317/1}
\address{Department of Mathematics, Harvard University, Cambridge, MA 02138, USA}
\email{gquick@math.harvard.edu}
\begin{document}
\begin{abstract}
We study profinite completion of spaces in the model category of profinite spaces and construct a rigidification of the completion functors of Artin-Mazur and Sullivan which extends also to non-connected spaces. Another new aspect is an equivariant profinite completion functor and equivariant  fibrant replacement functor for a profinite group acting on a space. This is crucial for applications where, for example, Galois groups are involved, or for profinite Teichm\"uller theory where equivariant completions are applied. Along the way we collect and survey the most important known results about profinite completion of spaces.
\end{abstract}
\maketitle
\section{Introduction}
The use of profinite methods in homotopy theory arose in the context of applications to arithmetic and algebraic geometry in the work of Artin and Mazur \cite{artinmazur}, where they introduced an \'etale pro-homotopy type for schemes. Their motivation was to define \'etale topological invariants for schemes. It turned out that for applications a profinitely completed version of the \'etale type is more suitable or even necessary. Artin and Mazur defined the profinite completion $X\to \hat{X}$ of a connected space $X$ as the universal map from $X$ to objects in pro-$\Hh_{0, \mathrm{fin}}$, the pro-category of the homotopy category $\Hh_{0, \mathrm{fin}}$ of connected finite spaces, i.e. spaces $X$ whose homotopy groups $\pi_nX$ are finite for all $n\geq 0$ and are even trivial for almost all $n$. Artin and Mazur constructed $\hat{X}$ as a pro-object in the homotopy category of spaces.\\
In his fundamental work on homotopy theory and the Adams conjecture \cite{sullivan}, Sullivan revisited these methods and showed that this profinite completion as a pro-space can in fact be realized in the homotopy category of spaces $\Hh$ itself by proving that the Artin-Mazur pro-object  admits a homotopy limit in $\Hh$.\footnote{Sullivan's comment on page 2 of \cite{sullivan}: "... we had to domesticate the abstract beasts of \cite{artinmazur} to make them usable in ordinary algebraic topology."}
Later on, Rector rigidified the definition of Artin and Mazur and constructed a profinite completion of a connected space as a pro-object in the category of connected finite spaces.\\
By introducing a more suitable category, Morel opened a whole new perspective on the question. He significantly improved the previous results for the pro-$p$-completion of spaces based on a pro-$p$-model structure for any fixed prime $p$. He considered in \cite{ensprofin} the category $\hSh$ of simplicial objects of profinite sets, called profinite spaces. Profinite completion of sets induces a completion functor from $\Sh$ to $\hSh$. In order to obtain a pro-$p$-completion functor that is homotopy invariant and generalizes Artin-Mazur $p$-completion, Morel equipped $\hSh$ with a model structure in which weak equivalences are maps that induce isomorphisms in continuous $\Z/p$-cohomology. One of the main results of \cite{ensprofin} is that every profinite space $X$ is weakly equivalent in $\hSh$ to a limit of finite-$p$-spaces, i.e. spaces which have only a finite number of nontrivial homotopy groups each of which is a finite $p$-group. After composition with the functor $\Sh \to \hSh$, this rigidifies the constructions of Artin-Mazur and Sullivan for pro-$p$-completion and also generalizes the results of Rector in \cite{rector} to non-connected spaces. \\
The purpose of this paper is to generalize these ideas to the full profinite completion for non-connected spaces on the basis of a different model structure on $\hSh$. Moreover, we extend the results to profinite spaces with a continuous action by a profinite group.\\
In \cite{profinhom}, a model structure on $\hSh$ has been constructed in which weak equivalences are maps that induce isomorphisms on profinite fundamental groups and continuous cohomology with finite local coefficients. The fibrant replacement functor in this structure has not been made explicit in \cite{profinhom}. The first goal of this paper is to generalize the idea of Morel and Rector to construct an explicit fibrant replacement and thereby to show that every profinite space is weakly equivalent in $\hSh$ to a limit of finite spaces in the above sense. Composition with the functor $\Sh \to \hSh$ gives a rigidification and generalization of profinite completion of spaces.\\
We remark that although we do not treat these cases in this paper, the methods of \cite{profinhom} do not only work for the class of all finite groups but also for any choice of subclass of finite groups. The construction in this paper is also applicable for any such subclass. For example one could choose a set $L$ of primes and define a pro-$L$-model structure on $\hSh$. This would yield a rigid pro-$L$-finite completion functor, again generalizing Rector's approach. Note that the case of pro- $p$-completion is special in this context. The construction of the $\Z/p$-fibrant replacement in \cite{ensprofin} is based on the Bousfield-Kan $\Z/p$-completion of \cite{bouskan}. We do not know of any way to use this approach in the more general case of profinite or pro-$L$-completion if $L$ contains more than one prime. A very interesting comparison of pro-$p$- and Bousfield-Kan $\Z/p$-completion is given in \cite{goersscomp}.\\
Now let $G$ be a profinite group and let $\hShg$ be the category of simplicial objects in the category of profinite sets with a continuous $G$-action. We call the objects of $\hShg$ profinite $G$-spaces. In \cite{gspaces}, a model structure on $\hShg$ has been defined which is based on the underlying model structure on $\hSh$.\\
Such profinite $G$-spaces occur naturally when we look at the \'etale topological type of a scheme $X$ defined over a field $k$. Let $\ok$ be a separable closure of $k$ and $\Xok$ the base change of $X$ to $\ok$. The absolute Galois group $G=\Gal(\ok/k)$ acts by functoriality on the completed version $\hEt \Xok$ of the \'etale type of $\Xok$ and, with a little care, $\hEt \Xok$ can be viewed as an object in $\hShg$. \\ 
One of the fundamental operations for group actions is taking fixed points. Since taking fixed points is not homotopy invariant, one also considers the invariant version called homotopy fixed points. For a profinite $G$-space $X$, one would like to remember the continuity of the $G$-action. Therefore, one defines the continuous homotopy fixed point space $X^{hG}$ of $X$ to be the $G$-fixed points of the mapping space of continuous maps from $EG$ to $R_GX$, where $EG$ denotes a contractible profinite space with a free $G$-action and $R_GX$ is a fibrant replacement of $X$ in $\hShg$. Many mathematical problems involving a group action can be formulated in terms of homotopy fixed points. Hence to understand this fibrant replacement is of fundamental importance for continuous group actions.\\ 
The main result of this paper is that there is an explicit fibrant replacement functor for $\hShg$ based on the fibrant replacement functor for $\hSh$. 
In particular, we obtain that every profinite $G$-space is weakly equivalent in $\hShg$ to a limit of  continuous $G$-spaces each of which is also a finite space in the usual sense. \\ 
Moreover, let $|G|$ be the underlying abstract group of $G$ and let $\Sh_{|G|}$ be the category of simplicial objects in the category of $|G|$-sets. We will define a $G$-equivariant completion functor from $\Sh_{|G|}$ to $\hShg$. Then the explicit fibrant replacement functor in $\hShg$ will provide a $G$-equivariant rigid profinite completion functor for $|G|$-spaces extending the completion functors for spaces.\\
We would like to mention two interesting applications of profinite $G$-spaces in algebraic geometry that motivated this paper. Firstly, Boggi has formulated in \cite{boggi} and \cite{boggi2} some central conjectures in profinite Teichm\"uller theory using profinite spaces, profinite $G$-spaces and a slightly different $G$-equivariant completion functor. This has been further studied and reviewed by Lochak in \cite{lochak}. Since these conjectures are stated in terms of homotopy theoretical problems, an explicit fibrant replacement functor for profinite $G$-spaces seems crucial for this approach. For, the set-theoretic $G$-completion functor from $G$-spaces to profinite $G$-spaces just yields an object in $\hShg$. But if we want to understand the homotopy type of this object, we need a rigid and homotopy invariant $G$-completion. This is provided by the fibrant replacement functor in $\hShg$.\\ 
Secondly, continuous homotopy fixed points have been used in \cite{gspaces} to reinterpret the map of Grothendieck's section conjecture. Let $k$ be a number field $k$ and let $X$ be a smooth projective curve of genus at least $2$ over $k$. There is a short exact sequence of \'etale fundamental groups
$$
1 \to \pi_1^{\et}\Xok \to \pi^{\et}_1X\to \Gal(\ok/k) \to 1.
$$
Grothendieck's conjecture predicts that the map from the set of $k$-rational points of $X$ to the set of sections $\Gal(\ok/k) \to \pi^{\et}_1X$ up to conjugation by the action of $\pi_1^{\et} \Xok$ is a bijection. By generalizing a result of Cox, it has been shown in \cite{gspaces} that the latter set of sections is in fact in bijection with the set $\pi_0((\hEt \Xok)^{h\Gal(\ok/k)})$ of connected components of the continuous homotopy fixed points of $\hEt \Xok$ under the Galois action. If there is any hope to obtain some new information about the map of the section conjecture via this approach, one has to understand the shape of the $\Gal(\ok/k)$-equivariant fibrant replacement of $\hEt \Xok$ in $\hShg$.\\
Let us quickly outline the content of the paper. In the next section we introduce profinite spaces with the model structure on $\hSh$. We study the various types of profinite completion of spaces and construct the explicit fibrant replacement functor in $\hSh$. Along this way, we resume the relation between profinite completion of groups and spaces and give a survey of known results in the setting of profinite spaces.\\ 
Then we study profinite $G$-spaces under the action of a profinite group $G$. We introduce a $G$-equivariant completion functor from $G$-spaces and how it can be simplified when either $G$ satisfies the property of strong completeness, e.g. when $G$ is topologically finitely generated, or when the action is discrete. Finally, we explain how fibrant replacement functor in $\hSh$ yields a fibrant replacement in $\hShg$. \\ 
{\bf Acknowledgements}. I would like to express my gratitude to the organizers of the Kyoto conferences on Galois-Teichm\"uller theory and Arithmetic Geometry where I had the great opportunity to learn about profinite Teichm\"uller theory and related ideas. I'm especially grateful to Pierre Lochak for drawing my attention to this subject and many interesting discussions. I would also like to thank Mike Hopkins for helpful conversations.
%
%
\section{Spaces and profinite spaces}
The homotopy category of topological spaces has a good combinatorial model provided by simplicial sets. It is the category, denoted by $\Sh$, of simplicial objects in the category of sets. An object of $\Sh$ will be called a space. Instead of just looking at sets, one could consider simplicial objects in categories of sets with additional structure, e.g. simplicial objects in the category of groups or pro-$p$-groups etc. In this paper, we will study a topological condition and consider simplicial objects in the category of profinite sets with its limit topology. This category has first been studied by Morel in \cite{ensprofin}.\\ 
Let $\Eh$ denote the category of sets and let $\Fh$ be the full subcategory of finite sets. Let $\hEh$ be the category of compact Hausdorff and totally disconnected topological spaces. We may identify $\Fh$ with a full subcategory of $\hEh$ in the obvious way. The limit functor $\lim$: pro-$\Fh \to \hEh$ is an equivalence of categories. Moreover, the forgetful functor $\hEh \to \Eh$ admits a left adjoint $\compl:\Eh \to \hEh$ which is called profinite completion. For a set $X$, its profinite completion is defined as follows. Let $\Rh(X)$ be the set of equivalence relations on $X$ such that $X/R$ is a finite set. The set $\Rh(X)$ is ordered by inclusion. The profinite completion of $X$ is defined as the limit of the finite sets $X/R$ over all $R\in \Rh(X)$, i.e. $\hat{X}:=\lim_{R\in \Rh(X)} X/R$.\\ 
We denote by $\hSh$ the category of simplicial profinite sets, i.e. simplicial objects in $\hEh$. The objects of $\hSh$ will be called {\em profinite spaces} and will be our main object of study. The reader should note that, although $\hEh$ is equivalent to pro-$\Fh$, $\hSh$ is not equivalent to the pro-category of simplicial finite sets. A careful treatement of such phenomena is given in \cite{calclim}. 
If $X$ is a profinite space, we can decompose it as a limit of simplicial finite sets. We define the set $\Rh(X)$ of simplicial \emph{open} equivalence relations on $X$. An element $R$ of $\Rh(X)$ is a simplicial profinite subset of the product $X\times X$ such that, in each degree $n$, $R_n$ is an equivalence relation on $X_n$ and an open subset of $X_n\times X_n$. It is ordered by inclusion. For every element $R$ of $\Rh(X)$, the quotient $X/R$ is a simplicial finite set and the map $X \to X/R$ is a map of profinite spaces. The canonical map $X \to \lim_{R\in \Rh(X)} X/R$ is an isomorphism in $\hSh$, cf. \cite{ensprofin}, Lemme 1.\\
The completion of sets induces a functor $\compl:\Sh \to \hSh$, which is also called profinite completion. For a space $Z$, its profinite completion can be described as follows. Let $\Rh(Z)$ be the set of all simplicial equivalence relations $R$ on $Z$ such that the quotient $Z/R$ is a simplicial finite set, i.e. each $Z_n/R_n$ is a finite set for $n\geq 0$. Then $\Rh(Z)$ is again ordered by inclusion. The profinite completion $\hat{Z}$ of $Z$ is defined as the limit of the $Z/R$ for all $R\in \Rh(Z)$, i.e. $\hat{Z} := \lim_{R\in \Rh(Z)} Z/R$. Profinite completion of spaces is again left adjoint to the forgetful functor $|\cdot|:\hSh \to \Sh$ which sends a profinite space to its underlying simplicial set. \\
The category $\hSh$ can be equipped with different interesting model structures. Morel was the first to define a model structure based on continuous $\Z/p$-cohomology in \cite{ensprofin}. We will use another model structure that has been defined in \cite{profinhom} and that we will introduce now.\\
Let $X$ be a profinite space and let $\pi$ be a topological abelian group. The continuous cohomology $H_{\mathrm{cts}}^{\ast}(X;\pi)$ of $X$ with coefficients in $\pi$ is defined as the cohomology of the complex $C_{\mathrm{cts}}^{\ast}(X;\pi)$ of continuous cochains of $X$ with values in $\pi$, i.e. $C_{\mathrm{cts}}^n(X;\pi)$ denotes the set $\Hom_{\hat{\Eh}}(X_n,\pi)$ of continuous maps $\alpha:X_n \to \pi$ and the differentials $\delta^n:C_{\mathrm{cts}}^n(X;\pi)\to C_{\mathrm{cts}}^{n+1}(X;\pi)$ are the morphisms associating to $\alpha$ the map $\sum_{i=0}^{n+1}\alpha \circ d_i$, where $d_i$ denotes the $i$th face map of $X$. If $\pi$ is a finite abelian group and $Z$ a simplicial set, then the cohomologies $H^{\ast}(Z;\pi)$ and $H_{\mathrm{cts}}^{\ast}(\hat{Z};\pi)$ are canonically isomorphic by adjointness of profinite completion of sets and forgetful functor, cf. \cite{ensprofin}. \\
If $\Gamma$ is an arbitrary profinite group, we may still define the first cohomology of $X$ with coefficients in $\Gamma$ as done by Morel in \cite{ensprofin}, p. 355. The functor $X\mapsto \Hom_{\hEh}(X_0,\Gamma)$
is represented in $\hSh$ by a profinite space $E\Gamma$ given in degree $n$ by $E\Gamma_n=\Gamma^{n+1}$, the $n+1$-fold product of $\Gamma$. We define the $1$-cocycles $Z_{\mathrm{cts}}^1(X;\Gamma)$ to be the set of continuous maps $f:X_1 \to \Gamma$ such that $f(d_0x)f(d_2x)=f(d_1x)$ for every $x \in X_1$. The functor $X\mapsto Z_{\mathrm{cts}}^1(X;\Gamma)$ is represented by a profinite space $B\Gamma=E\Gamma/\Gamma$. Furthermore, there is a map $\delta:\Hom_{\hSh}(X,E\Gamma) \to Z_{\mathrm{cts}}^1(X;\Gamma)\cong \Hom_{\hSh}(X,B\Gamma)$ which sends $f:X_0 \to \Gamma$ to the $1$-cocycle $x\mapsto \delta f(x)=f(d_0x)f(d_1x)^{-1}$. We denote by $B_{\mathrm{cts}}^1(X;\Gamma)$ the image of $\delta$ in $Z_{\mathrm{cts}}^1(X;\Gamma)$ and we define the pointed set $H_{\mathrm{cts}}^1(X,\Gamma)$ to be the quotient $Z_{\mathrm{cts}}^1(X;\Gamma)/B_{\mathrm{cts}}^1(X;\Gamma)$. Finally, we define $\pi_0X$ to be the coequalizer in $\hEh$ of the diagram $d_0,d_1:X_1 \rightrightarrows X_0$.\\
The profinite fundamental group of $X$ is defined via covering spaces in the spirit of Grothendieck, see \cite{profinhom}. There is a universal profinite covering space $(\tilde{X},x)$ of $X$ at a vertex $x \in X_0$. Then $\pi_1(X,x)$ is defined to be the group of automorphisms of $(\tilde{X},x)$ over $(X,x)$. It has a natural structure of a profinite group as the limit of the finite automorphism groups of the finite Galois coverings of $(X,x)$.\\ 
The collection of vertices and morphisms between covering spaces over different vertices of $X$ defines a profinite fundamental groupoid $\Pi X$.\footnote{We will discuss the notion of profinite groupoids in the third section of the paper, in particular \S3.3 and \S3.4. The reader could jump to this section for precise definitions.} A profinite local coefficient system $\Mh$ on $X$ is a functor from $\Pi X$ to profinite abelian groups such that the action of $\pi_1(X,x)$ on $\Mh(x)$ is continuous. The cohomology of $X$ with coefficients in $\Mh$ is then defined as the cohomology of the complex $\hom_{\Pi X}((\tilde{X},-),\Mh)$ of continuous natural transformations. For any further details, we refer the reader to \cite{profinhom}. 
\begin{defn}\label{defnwe}
A morphism $f:X\to Y$ in $\hSh$ is called,\\
{\rm (1)} a {\em weak equivalence} if the induced map $f_{\ast}:\pi_0(X) \to \pi_0(Y)$ is an isomorphism of profinite sets, $f_{\ast}:\pi_1(X,x) \to \pi_1(Y,f(x))$ is an isomorphism of profinite groups for every vertex $x\in X_0$ and $f^{\ast}:H_{\mathrm{cts}}^q(Y;\Mh) \to H_{\mathrm{cts}}^q(X;f^{\ast}\Mh)$ is an isomorphism for every local coefficient system $\Mh$ of finite abelian groups on $Y$ for every $q\geq 0$;\\
{\rm (2)} a {\em cofibration} if $f$ is a level-wise monomorphism;\\
{\rm (3)} a {\em fibration} if it has the right lifting property with respect to every cofibration that is also a weak equivalence. 
\end{defn}
The following theorem has been stated in \cite{profinhom} (with the slight correction that in \cite{profinhom} the generating sets of fibrations and cofibrations had been chosen too small; the revised proof is given in \cite{gspaces}, Theorem 2.3). 
\begin{theorem}\label{modelstructure}
The above defined classes of weak equivalences, cofibrations and fibrations provide $\hSh$ with the structure of a fibrantly generated left proper simplicial model category. We denote the homotopy category by $\hHh$.
\end{theorem}
We consider the category $\Sh$ of simplicial sets with the usual model structure of \cite{homalg}. We denote its homotopy category by $\Hh$. Then the next result follows as in \cite{profinhom}, Proposition 2.28.
\begin{prop}\label{adjcompletion}
{\rm 1.} The completion functor $\compl: \Sh \to \hSh$ preserves weak equivalences and cofibrations.\\
{\rm 2.} The forgetful functor $|\cdot|:\hSh \to \Sh$ preserves fibrations and weak equivalences between fibrant objects.\\
{\rm 3.} The induced completion functor $\compl: \Hh \to \hHh$ and the right derived functor $R|\cdot|:\hHh \to \Hh$ form a pair of adjoint functors.
\end{prop}
Let $R_f$ be a fibrant replacement functor in $\hSh$ which exists by the general nonsense of fibrantly generated model structures.
\begin{defn}\label{defhomotopygroups}
Let $X$ be a pointed profinite space. We define the {\em $n$th profinite homotopy group of $X$} for $n \geq 2$ to be the profinite group
$$\pi_n(X):=\pi_0(\Omega^n(R_fX)).$$
\end{defn}
The homotopy groups in Definition \ref{defhomotopygroups} carry a natural profinite structure. In order to compute the homotopy groups of a profinite space $X$ we take a fibrant replacement of $X$ in $\hSh$, and then take the usual homotopy groups of the fibrant simplicial set $R_fX$.  
 The main goal of this paper is to get a better understanding of this fibrant replacement in $\hSh$. This is the task for the next section. \\
But let us remark that, after defining weak equivalences and constructing the model structure on $\hSh$, one gets that these homotopy groups detect weak equivalence, i.e. a map $f$ in $\hSh$ is a weak equivalence of connected profinite spaces if and only if $\pi_*(f)$ is an isomorphism of profinite (abelian) groups.  
%
%
\section{Profinite completion, fibrant replacements and pro-finite spaces}
\subsection{Profinite completion revisited}
In the previous section, we have defined a set-theoretic profinite completion functor that sends a space to a profinite space and we have seen that every profinite space can be decomposed as a limit of simplicial finite sets, i.e. simplicial objects of $\Fh$. But from a homotopy theoretic point of view this functor alone is not satisfactory. A \emph{finite space} in homotopy theory is a simplicial set $X$ that has only finitely many non-trivial homotopy groups and those being non-trivial are finite groups. On the other side, a simplicial finite set and even a finite simplicial set, i.e. one with only finitely many non-degenerate simplices, may have infinite homotopy groups. So being a profinite space in our terminology does not imply that the homotopy groups of the underlying simplicial set are profinite. But a profinite completion in terms of homotopy theory should be a functor that sends a space $X$ in a universal way to a limit $\hat{X}_f=\lim_i X_i$ or a cofiltering system of finite spaces $X_i$ such that each homotopy group of $\hat{X}_f$ (as a simplicial set) is the limit of the finite homotopy groups of the $X_i$. 
\begin{remark}
The reader should be aware of the overloaded terminology which might be confusing at first glance. We stick to the common and well-known notion of a \emph{finite space} for a simplicial set $X$ that has only finitely many non-trivial homotopy groups all of which are finite groups. The reader should not confuse it with the notion of a \emph{profinite space} for an object in $\hSh$, even when the profinite space happens to be a simplicial finite set. But we will see in \S 3.6 that every profinite space is weakly equivalent in $\hSh$ to a pro-object of finite spaces. Hence in the end the two notions are related in the way that one may expect.
\end{remark}
\begin{example}\label{circle}
The simplicial circle $S^1=\Delta^1/\partial \Delta^1$ is a simplicial finite set. But its fundamental group $\pi_1S^1=\Z$ is infinite. The point is that $S^1$ is not fibrant. So if we want to calculate its homotopy groups we have to replace it by a weakly equivalent fibrant simplicial set, for example the classifying space $B\Z$. The profinite fundamental group of $S^1$ as an object in $\hSh$ is $\hat{\Z}$, the profinite completion of $\Z$. One way to calculate this fundamental group is to replace $S^1$ by a profinite space that is fibrant in $\hSh$ and weakly equivalent to $S^1$ in $\hSh$, for example $B\hat{\Z}$. The set-theoretic profinite completion $\Sh \to \hSh$ functor of the previous section sends $S^1$ to itself because it is already a simplicial finite set. The homotopy-theoretic completion functor should rather send $S^1$ to $B\hat{\Z}=\lim_n B\Z/n$.
\end{example}
The first solution for the existence of a profinite completion functor, in the pro-homotopy category of connected spaces, has been given by Artin and Mazur in \cite{artinmazur}. Let us quickly recall their construction. 
For a category $\Ch$ with small limits, the pro-category of $\Ch$, denoted pro-$\Ch$, has as objects all cofiltering diagrams $X:I \to \Ch$. Its sets of morphisms are defined as
$$\Hom_{\pro-\Ch}(X,Y):=\lim_{j\in J}\colim_{i\in I} \Hom_{\Ch}(X_i,Y_j).$$
The functor sending an object $X$ of $\Ch$ to the constant pro-object with value $X$ makes $\Ch$ a full subcategory of pro-$\Ch$. The right adjoint of this embedding is the limit functor $\lim$: pro-$\Ch$ $\to \Ch$, which sends a pro-object $X$ to the limit in $\Ch$ of the diagram corresponding to $X$.\\ 
Let $\Hh_0$ denote the subcategory of connected spaces and $\Hhf$ be the subcategory of connected finite spaces in $\Hh$. Artin and Mazur showed in \cite{artinmazur} that, for every space $X\in \Hh_0$, the functor 
$$\Hhf \to \Eh,~F\mapsto [X,F],$$ 
is pro-representable in $\Hhf$. The representing pro-object $\hat{X}^{\mathrm{AM}}\in \mathrm{pro}-\Hhf$ is called the (Artin-Mazur) profinite completion of $X$. 
Then Sullivan showed in \cite{sullivan} that the underlying diagram in $\Hhf$ of $\hat{X}^{\mathrm{AM}}$ has a limit $\hat{X}^{\mathrm{Su}}$ in $\Hh$. Moreover, there are analogues of these functors for various subclasses of the class of finite groups, for which one replaces $\Hhf$ by its subcategory of connected finite spaces whose homotopy groups are all in this smaller class. For example, one could consider the class of finite $p$-groups for a fixed prime number $p$. \\
The drawback of these constructions is that they are only obtained in the homotopy category or even its pro-category. A first rigidification of $\hat{X}^{\mathrm{AM}}$ has been given by Rector in \cite{rector}. For a connected space $X$, he defined a rigid pro-space that is weakly equivalent in pro-$\Hh$ to $\hat{X}^{\mathrm{AM}}$.\\ 
We will generalize this rigidification to arbitrary spaces and reinterpret it in the language of model categories via the model structure of Theorem \ref{modelstructure} on the category $\hSh$. 
As Example \ref{circle} already suggests, a good profinite completion functor from the homotopy point of view can be obtained by composing the set-theoretic completion $\Sh \to \hSh$ with a functorial fibrant replacement in $\hSh$. So the model structure of Theorem \ref{modelstructure} is the crucial ingredient in order to get a rigid version of the full profinite completion functor.\\ 
The idea to use $\hSh$ as the category in which a pro-$p$-completion should take place is due to Morel. He paved the way. Morel proved in \cite{ensprofin} that there is a model structure on $\hSh$ for each prime number $p$ in which the weak equivalences are maps that induce isomorphisms on $\Z/p$-cohomology. The fibrant replacement functor $R_p$ of \cite{ensprofin} yields a rigid version of the pro-$p$-finite-completion of Artin-Mazur and Sullivan. The homotopy groups for this structure are pro-$p$-groups being defined as above using $R_p$. 
\begin{remark}
Let $L$ be a set of primes. A finite $L$-group is a finite group whose order is only
divisible by primes of $L$. The methods of \cite{profinhom} and \cite{gspaces} provide a framework for a rigid model of the Artin-Mazur and Sullivan pro-$L$-completion. In this paper, we will always treat the case of the class of all finite groups. All statements on profinite completion can be transferred to pro-$L$-completion by rephrasing the proofs in terms of finite $L$-groups instead of all finite groups.
\end{remark}
%
%
\subsection{Completion of spaces versus completion of groups I}
Before we start with the construction, we consider the question in which way the set-theoretic completion of spaces interacts with the well-known profinite completion of groups. Since this question will be a constant companion, we should be aware of the problem and should be equipped with some terminology. \\
We will also denote the profinite completion of a group $G$ by $\hat{G}$. (The context will always make clear which completion is applied.) It is defined as the limit $\lim_U G/U$ over all open normal subgroups of $G$. It is equipped with a natural map $G\to \hat{G}$ which is universal among maps from $G$ to profinite groups. \\ 
Given a pointed space $X\in \Shp$, the homotopy groups of Definition \ref{defhomotopygroups} of its profinite completion $\hat{X}\in \hShp$ are profinite groups. Hence the induced map $\pi_tX \to \pi_t\hat{X}$ factors through the group completion of $\pi_tX$, i.e. there is a commutative diagram
$$\xymatrix{
\pi_tX\ar[d] \ar[r] &\pi_t\hat{X}\\ 
 \widehat{\pi_tX}\ar[ur]_{\varphi_t}.}
$$
It is a fundamental question how the completions of spaces and of groups interact. For fundamental groups, we have the following result, which follows from the construction of the profinite fundamental group of profinite spaces via finite covering spaces, see \cite{profinhom} \S\,2.1.
\begin{prop}\label{profinitecompletion}
Let $X$ be a connected pointed simplicial set $X$. The profinite group $\pi_1(\hat{X})$ is equal to the group completion of $\pi_1(X)$, i.e. $\varphi_1:\widehat{\pi_1(X)}\stackrel{\cong}{\to} \pi_1(\hat{X})$ is an isomorphism of profinite groups. 
\end{prop}   
For higher homotopy groups, the profinite Hurewicz theorem implies the next result, see \cite{profinhom} Proposition 2.31.
\begin{prop}\label{hatpin}
Let $X$ be a pointed simplicial set. Suppose that $\pi_q(X)=0$ for $q<n$. Then $\pi_n(\hat{X})$ is the profinite group completion of $\pi_n(X)$.
\end{prop} 
Unfortunately, $\varphi_t$ is not an isomorphism in general for $t\geq 2$. A related phenomenon is well known for group completion and cohomology. In \cite{serre}, this led Serre to call an abstract group $G$ {\em good} if the induced map $\psi:H_{\mathrm{cts}}^*(\hat{G};M)\to H^*(G;M)$ between continuous and discrete group cohomology is an isomorphism for every finite discrete $\hat{G}$-module $M$. It turns out that the notion of a good group is also crucial for the completion of spaces and its homotopy groups.\\
Let $G$ be an abstract group and let $BG$ be the simplicial classifying space of $G$ given in degree $n$ by a product of $n$ copies of $G$. It has a profinite analogue $B\hat{G}$ in $\hSh$ which is given in degree $n$ by the product of $n$ copies of the profinite group $\hat{G}$. We denote by $\widehat{BG} \in \hSh$ the (set-theoretic) profinite completion of the space $BG$. The universal property of profinite completion of spaces induces a commutative diagram 
$$\xymatrix{
BG \ar[r] \ar[d] & B\hat{G}  \\
\widehat{BG} \ar[ur]_{\varphi} & }
$$
in $\hSh$. As one might expect, the map $\varphi:\widehat{BG} \to B\hat{G}$ is in general neither an isomorphism nor a weak equivalence. The difference between the two spaces comes from the difference of the completion of $G$ as a set and its completion as a group. This difference is exactly on what Serre's notion of good groups is based on. For classifying spaces it can be rephrased as follows, see also \cite{artinmazur} \S 6. 
\begin{prop}\label{cor6.6}
The canonical map $\varphi:\widehat{BG} \to B\hat{G}$ of profinite spaces is a weak equivalence in $\hSh$ if and only if $G$ is good.
\end{prop}
\begin{proof}
We have seen that $\pi_1(\widehat{BG})\cong\hat{G}$ is the group completion of $\pi_1(BG) \cong G$. Moreover, the profinite fundamental group of $B\hat{G}$ is also equal to $\hat{G}$ and $\pi_1(\varphi)$ is an isomorphism.\\
The crucial point where the properties of $G$ come into play is the question wether $\varphi$ induces an isomorphism in cohomology. For every finite $G$-module $M$, $\varphi$ induces the sequence of maps 
$$H^q(G;M)\cong H^q(BG;M)\cong H_{\mathrm{cts}}^q(\widehat{BG};M) \to H_{\mathrm{cts}}^q(B\hat{G};M)\cong H_{\mathrm{cts}}^q(\hat{G};M)$$
between the usual group cohomology $H^q(G;M)$ and the continuous cohomology $H_{\mathrm{cts}}^q(\hat{G};M)$. This map is an isomorphism for every $q$ if and only if $G$ is good. 
\end{proof}

The same holds for Eilenberg-MacLane spaces $K(G,r)$ for $r>1$ and an abelian group $G$, i.e. the canonical map $\widehat{K(G,r)} \to K(\hat{G},r)$ is a weak equivalence in $\hSh$ if and only if $G$ is good. The proof of this statement is more complicated than the previous one, see \cite{artinmazur} \S 6.\\
We will see below that by modifying slightly the notion of good groups by considering the action of the fundamental group on the higher homotopy groups, one obtains a sufficient condition such that the completion of spaces commutes with the one of groups for all homotopy groups. This result is due to Sullivan \cite{sullivan} and we will translate it to our setting via the following fibrant replacement functor in $\hSh$.
\subsection{Simplicial groupoids}
The classifying space functor for groups given by the bar construction has a natural analogue $\bar{W}$ for simplicial groups, i.e. simplicial objects in the category of groups. If $\Gamma$ is a simplicial group, let $W\Gamma$ be the simplicial set with
$$(W\Gamma)_n = \Gamma_n \times \Gamma_{n-1} \times \ldots \times \Gamma_0.$$
Then $W\Gamma$ becomes a $\Gamma$-space if we define $\Gamma \times W\Gamma \to W\Gamma$ by:
$$(h_n, (g_n, g_{n-1}, \ldots, g_0)) \mapsto (h_ng_n, g_{n-1}, \ldots, g_0)$$
for $h_n\in \Gamma_n$. The classifying space $\bar{W}\Gamma$ is defined as the quotient of $W\Gamma$ by the left $\Gamma$-action. In degree $0$, $\bar{W}\Gamma_0$ has just one element, so it is a reduced space, and in degree $n$ it is given by 
$$(\bar{W}\Gamma)_n = \Gamma_{n-1} \times \ldots \times \Gamma_0.$$
The functor $\bar{W}$ from simplicial groups to reduced spaces has a left adjoint, the free loop group construction $\Gamma$. For a reduced space $X$, i.e. $X_0$ consists of a single vertex, $\Gamma X$ is the simplicial group given in degree $n$ by the free group on the set $X_{n+1} - s_0(X_n)$. In fact, the pair of functors $\bar{W}$ and $\Gamma$ induce an equivalence between the homotopy categories of simplicial groups and of reduced spaces, cf. \cite{gj}, V Corollary 6.4.\\
In \cite{dk}, Dwyer and Kan extended this equivalence of homotopy categories to the homotopy category of all spaces by considering simplicial groupoids instead of just simplicial groups. A groupoid $\Gamma$ is a small category in which all maps are invertible. For an object $x$ of $\Gamma$, we denote the set of automorphisms of $x$ by $\Gamma(x,x)$. Dwyer and Kan define a simplicial groupoid to be a simplicial object in the category of groupoids whose object sets are all equal to a given set of objects. In other words, a simplicial groupoid $\Gamma$ consists of groupoids $\Gamma_n$ for every $n\geq 0$ and a functor $\Gamma_m \to \Gamma_n$ for every ordinal number map $\theta:[n] \to [m]$ such that all sets of objects $\mathrm{Ob}(\Gamma_n)$ are equal to one set of objects $\mathrm{Ob}(\Gamma)$ and all the functors $\Gamma_m \to \Gamma_n$ induce the identity map on $\mathrm{Ob}(\Gamma)$. We denote the category of simplicial groupoids by $s\mathrm{Gd}$. We refer the reader to \cite{gj}, V \S 7, for a careful discussion on simplicial groupoids. \\
There is a classifying space functor $\bar{W}:s\mathrm{Gd} \to \Sh$ that extends the one on simplicial groups. If $\Gamma$ is a simplicial groupoid, the vertices of $\bar{W}\Gamma$ are the objects of $\Gamma$ and, for $n\geq 1$, $\bar{W}\Gamma_n$ is given by the set of sequences of maps in $\Gamma$
$$Y_n \stackrel{g_{n-1}}{\longrightarrow} \ldots \stackrel{g_1}{\longrightarrow} Y_1 \stackrel{g_{0}}{\longrightarrow} Y_0$$
where each $g_i$ is a morphism in $\Gamma_i$.\\ 
%
%
The classifying space functor $\bar{W}$ is right adjoint to the loop groupoid functor $\Gamma$. For a simplicial set $X$, the loop groupoid $\Gamma X$ on $X$ is the simplicial groupoid whose object set is the set of vertices of $X$ and whose morphisms are in degree $n$ given by the free groupoid on generators $[x]: x_1 \to x_0$ with $x \in X_{n+1}$, subject to the relations $s_0 x_0 = 1_{x_0}$, $x_0 \in X_n$. The face and degeneracy maps are defined in \cite{gj}, V \S 7.\\ 
Applying $\Gamma$ and then $\bar{W}$ to a space $X$ yields a space $\bar{W}\Gamma X$. Every $n$-simplex $x$ of $X$ determines a sequence of morphisms 
$$x_n \stackrel{[d_0^{n-1}x]}{\longrightarrow} \ldots \stackrel{[d_0x]}{\longrightarrow} x_1 \stackrel{[x]}{\longrightarrow} x_0$$
in $\Gamma X$. This defines a canonical map of spaces
$$\eta: X \to \bar{W}\Gamma X.$$
%
%
Dwyer and Kan show that $s\mathrm{Gd}$ has an important model structure. The consequence of their theorem that motivates our construction is that $\eta$ is a weak equivalence and $\bar{W}\Gamma X$ is a fibrant model of $X$ in $\Sh$. 
\subsection{Simplicial profinite groupoids}
We would like to extend these ideas to simplicial profinite groupoids. Let us first recall the construction of free profinite groups on a profinite set,  cf. for example \cite{ribes} \S 3.3.\\
The free profinite group on a profinite set $S$ is a profinite group $\hat{F}(S)$ equipped with a canonical continuous injection $\iota: S \to \hat{F}(S)$ which topologically generates $\hat{F}(S)$, i.e. $\hat{F}(S)=\overline{\langle \iota(S) \rangle}$, and satisfies the following universal property:\\
For any continuous map $\varphi: S \to H$ to a profinite group $H$ such that $\varphi(S)$ generates $H$ topologically, i.e. $H=\overline{\langle \varphi(S) \rangle}$, there is a unique continuous homomorphism $\bar{\varphi}:\hat{F}(S) \to H$ such that the diagram
$$
\xymatrix{
S \ar[d]_{\iota} \ar[dr]^{\varphi} \\
\hat{F}(S) \ar[r]_{\bar{\varphi}} & H}
$$
commutes.\\ 
 If $S$ is a finite set, then the free profinite group $\hat{F}(S)$ can be constructed by taking the free abstract group $F(S)$ on $S$ and then forming the profinite group completion of $F(S)$, i.e. $\hat{F}(S) = \widehat{F(S)}$.
If $S$ is a profinite set given as an inverse limit $S=\lim_i S_i$ of finite sets $S_i$, then the free profinite group $\hat{F}(S)$ on $S$ is not just the profinite completion of the abstract free group on the underlying set $S$. But Ribes and Zalesskii show in \cite{ribes}, Proposition 3.3.9, that $\hat{F}(S)$ can be constructed as 
$$\hat{F}(S) = \lim_i \hat{F}(S_i).$$
Now let $X$ be a reduced profinite space. We define its free profinite simplicial loop group to be the simplicial profinite group $\hat{\Gamma}X$, i.e. simplicial object in the category of profinite groups, that is given in degree $n$ by the free profinite group on the profinite set $X_{n+1} - s_0(X_n)$. 
The classifying space functors $W$ and $\bar{W}$ are defined for simplicial profinite groups in exactly the same way as above for simplicial groups. The only difference is that they have values in the category of profinite spaces. Moreover, $\bar{W}$ is the natural right adjoint to the free profinite loop group functor $\hat{\Gamma}$.\\
For a non-reduced profinite space, we need a notion of a free simplicial profinite groupoid on $X$ that extends the free simplicial groupoid on spaces. Therefore, we make the following definitions, see also \cite{ensprofin}, p. 367.
We call a groupoid $\Gamma$ finite, if the set of objects of $\Gamma$ is finite and, for each object $x$, the set of automorphisms $\Gamma(x,x)$ of $x$ is a finite group. We call $\Gamma$ a profinite groupoid if the set of objects of $\Gamma$ is profinite and, for each object $x$, the set of automorphisms $\Gamma(x,x)$ of $x$ is a profinite group. The profinite completion $\hat{\Gamma}$ of a groupoid $\Gamma$ is the limit as a groupoid of the filtered system of its quotient groupoids which are finite groupoids.\\
A simplicial profinite groupoid is a simplicial groupoid $\Gamma$ in the above sense such that each $\Gamma_n$ is a profinite groupoid. The profinite completion functor from simplicial groupoids to simplicial profinite groupoids is defined by forming the profinite completion of groupoids in each dimension.\\
If $X$ is a simplicial finite set, we define the free profinite groupoid $\hat{\Gamma}X$ on $X$ as the profinite completion of the free groupoid on $X$. If $X$ is a profinite space, $X$ is canonically isomorphic to the limit of simplicial finite sets $\lim_R X/R$. We define the free profinite groupoid $\hat{\Gamma}X$ on $X$ as 
$$\hat{\Gamma}X:= \lim_R \hat{\Gamma}(X/R).$$
The classifying space functor $\bar{W}$ on simplicial profinite groupoids is defined in the same way as for simplicial groupoids. If $\Gamma$ is a simplicial profinite groupoid, then $\bar{W}\Gamma$ is a profinite space. The canonical map $\eta$ defined for spaces, has a profinite analogue $\heta$ for any profinite space $X$. It is defined as the limit of maps 
$$\heta/R: X/R \to \bar{W}\Gamma X/R $$ 
for each open equivalence relation $R$ on $X$, i.e. 
$$\heta=\lim_R \heta/R: X\to\bar{W}\hat{\Gamma}X.$$
It is possible to define a model structure on simplicial profinite groupoids as for simplicial groupoids and show analogues of Theorem 2.5 and Theorem 3.3 of \cite{dk}. But for our purposes we need only a small piece of the cake. The significant fact for us is that these constructions produce fibrant profinite spaces. 
\begin{prop}\label{sgfib}
Let $\Gamma$ be a simplicial profinite group.\\ 
(1) The underlying profinite space of $\Gamma$ is fibrant in $\hSh$.\\
(2) The profinite spaces $W\Gamma$ and $\bar{W}\Gamma$ are fibrant in $\hSh$.\\ 
(3) The quotient map $W\Gamma \to \bar{W}\Gamma$ and every principal $\Gamma$-bundle map is fibration in $\hSh$.
\end{prop}
\begin{proof}
All the assertions follow from the decomposition of simplicial objects into their tower of coskeleta and the description of the generating sets of fibrations $P$ and trivial fibrations $Q$ in the model structure of $\hSh$ given in \cite{gspaces}, p. 1027. The $n$th coskeleton $\cosk_nY$ of a profinite space $Y$ is given in degree $m$ by 
$$(\cosk_nY)_m= \Hom(\sk_n\Delta^m,Y)= \lim_{[k] \to [m],k\leq n} Y_k.$$
We observe that when $Y$ is a profinite space, the usual construction of the coskeleton of $Y$ inherits a natural profinite structure from $Y$. Moreover, the profinite space $Y$ is isomorphic in $\hSh$ to the limit $\lim_n \cosk_n Y$ of its coskeleta. Since the limit of a tower of fibrations is again a fibration, the map $Y\to *$ is a fibration if the maps $\cosk_{n+1}Y \to \cosk_nY$ for every $n\geq 2$ and $\cosk_2Y\to *$ are fibrations. \\
Recall from \cite{gspaces}, Theorem 2.3, that the prototype of a fibration in $\hSh$ is the canonical map 
$$E\pi \times_{\pi} L(M,n) \to E\pi \times_{\pi} K(M,n)$$
of homotopy orbits under a finite group $\pi$ where $M$ is a finite $\pi$-module and $L(M,n)=WK(M,n)$ is the contractible space defined above associated to the simplicial group $K(M,n)$. This extends immediately to a profinite group $\pi$ and a continuous profinite $\pi$-module $M$ using Proposition \ref{wepf} below and the lifting property as in the proof of Theorem 2.3 in \cite{gspaces}.\\
So let us start with a simplicial \emph{finite} group $\Gamma$ and consider the simplicial finite set $\bar{W}\Gamma$. We know that the underlying simplicial finite set of $\bar{W}\Gamma$ is fibrant in $\Sh$ and that the collection $\{\cosk_n \bar{W}\Gamma \}_n$ is a Postnikov tower for $\bar{W}\Gamma$. The homotopy groups of $\bar{W}\Gamma$ satisfy
$$\pi_{n+1}\bar{W}\Gamma=\pi_{n} \Gamma$$
for $n \geq 0$. These groups are finite groups, since $\Gamma$ is a simplicial finite group. We recall that $\pi_n\Gamma$ is equal to the $n$th homology $H_n(N\Gamma)$ of the normalized complex $N\Gamma$, $N\Gamma_n=\cap_{i=1}^{n}\Ker(d_i:\Gamma_n\to \Gamma_{n-1}$), which is a complex of finite groups with finite homology groups. Let $\pi_i=\pi_i\bar{W}\Gamma$ denote the $i$th homotopy group of $\bar{W}\Gamma$. The abelian group $\pi_n$ is a $\pi_1$-module for every $n\geq 2$. Moreover, $\bar{W}\Gamma$ is a minimal fibrant space and \cite{gj}, V Corollary 5.13, shows that for every $n\geq 2$ there is a pullback square 
\begin{equation}\label{coskpb}
\xymatrix{
\cosk_{n+1}\bar{W}\Gamma \ar[d] \ar[r] & E\pi_1 \times_{\pi_1} L(\pi_n,n+1) \ar[d]^q \\
\cosk_n \bar{W}\Gamma \ar[r]_{k_n} & E\pi_1 \times_{\pi_1} K(\pi_n,n+1).}
\end{equation}
The map $k_n$ is called the $k$-invariant. It fits into a commutative diagram
\begin{equation}\label{kn}
\xymatrix{
\cosk_{n+1}\bar{W}\Gamma \ar[d] \ar[r] & K(\pi_n,n+1) \ar[d] \\
\cosk_n\bar{W}\Gamma \ar[r ]\ar[ur]_{k_n} & B\pi_1.}
\end{equation}
Viewed as a relative cocycle in $Z^{n+1}(\cosk_n\bar{W}\Gamma; \pi_n)$ the map $k_n$ corresponds to the map that assigns to every $n+1$-simplex $\sk_n\Delta^{n+1}\to \bar{W}\Gamma$ the corresponding element in $\pi_n=\pi_n\bar{W}\Gamma$. Since 
$$q:E\pi_1 \times_{\pi_1} L(\pi_n,n+1) \to E\pi_1 \times_{\pi_1} K(\pi_n,n+1)$$ 
is a fibration in $\hSh$ and since fibrations are stable under pullbacks, we conclude that each $\cosk_{n+1}\bar{W}\Gamma \to \cosk_n \bar{W}\Gamma$ is a fibration in $\hSh$ for $n\geq 2$. Moreover, $\cosk_2\bar{W}\Gamma\to *$ is equal to $B\pi_1 \to *$, another generating fibration in $\hSh$. Thus $\bar{W}\Gamma$ is a fibrant object in $\hSh$.\\
Now let $\Gamma$ be a simplicial \emph{profinite} group. Then $\Gamma$ is isomorphic as a simplicial profinite group to the limit $\lim_U \Gamma/U$ of its simplicial finite quotient groups $\Gamma/U$ where $U$ runs through the simplicial profinite subgroups of $\Gamma$ such that $U_n$ is an open normal subgroup of $\Gamma_n$ for each $n$. Moreover, since $\Gamma$ is a simplicial \emph{group}, all the constructions we applied commute with this limit. For, as we have remarked above, $\cosk_n\bar{W}\Gamma$ is the limit of the $\cosk_n(\bar{W}(\Gamma/U))$. Moreover, the normalized complex $N\Gamma$ is the limit of the normalized complexes $N\Gamma/U$ and so is a complex of profinite groups. The homology $H_n(N\Gamma)$ commutes with this limit as well and we get that $\pi_n=\pi_n\bar{W}\Gamma=H_{n-1}(N\Gamma)$ is the limit of finite groups $\pi_n/U:=H_{n-1}(N\Gamma/U)$.
Hence the fundamental group $\pi_1$ is the inverse limit of the finite groups $\pi_1/U$ and the $\pi_1$-module $\pi_n$, for $n\geq 2$, is the limit of the finite $\pi_1/U$-modules $\pi_n/U$. This implies that each $\pi_n$ is a continuous profinite $\pi_1$-module. Hence the $k$-invariant $k_n$ becomes an element in the group of continuous cocycles and corresponds to a map $\cosk_n\bar{W}\Gamma \to K(\pi_n,n+1)$ over $B\pi_1$ in $\hSh$. So diagram (\ref{coskpb}) for the simplicial profinite group $\Gamma$ is in fact a diagram in $\hSh$. Since $q$ is again a fibration in $\hSh$, the map of profinite spaces $\cosk_{n+1}\bar{W}\Gamma \to \cosk_n \bar{W}\Gamma$ is a fibration in $\hSh$ for $n\geq 2$. Similarly, $\cosk_2\bar{W}\Gamma=B\pi_1\to *$ is a fibration in $\hSh$ and we conclude as above that $\bar{W}\Gamma$ is a fibrant object in $\hSh$.\\
A similar argument applied to the relative coskeleton functor shows that $W\Gamma \to \bar{W}\Gamma$ is a fibration in $\hSh$. Hence $W\Gamma$ is also a fibrant profinite space. Furthermore, $\bar{W}\Gamma$ classifies principal $\Gamma$-bundles in $\hSh$, i.e. every principal $\Gamma$-bundle $E\to B$ is a pullback of $W\Gamma \to \bar{W}\Gamma$ via some classifying map $B\to \bar{W}\Gamma$ in $\hSh$. Hence $E\to B$ is also a fibration in $\hSh$. In particular, the map $\Gamma \to *$ is a principal $\Gamma$-bundle and hence a fibration in $\hSh$.
\end{proof}

It remains to generalize this result to simplicial profinite groupoids. If $\Gamma$ is a profinite groupoid, a profinite module $\Mh$ over $\Gamma$ is a functor from $\Gamma$ to the category of abelian profinite groups such that the profinite group $\Gamma(x,x)$ acts continuously on the profinite abelian group $\Mh(x)$ for every object $x$ of $\Gamma$. Given such a module $\Mh$ and an integer $n\geq 0$, the Eilenberg-MacLane object $K(\Mh,n)$ is the profinite space which has as $k$-simplices the pairs $(u,v)$ such that $u$ is a $k$-simplex $x_0 \to x_1 \to \ldots \to x_k$ of the profinite nerve $B\Gamma$ and $v$ is a $k$-simplex of the profinite Eilenberg-MacLane space $K(\Mh(x_0),n)$, see \cite{dk2}, 1.2 (iv). There is the forgetful map $K(\Mh,n)\to B\Gamma$ in $\hSh$. 
An example is given by the profinite fundamental groupoid $\Pi_1X$ of a profinite space $X$. The objects of $\Pi_1X$ are the vertices of $X$ and the higher homotopy groups define a profinite module $\Pi_nX$ over $\Pi_1X$ defined by sending $x\in X_0$ to the profinite $\pi_1(X,x)$-module $\Pi_nX(x)=\pi_n(X,x)$.\\
\begin{prop}\label{wepf}
A map $f:X\to Y$ is a weak equivalence in $\hSh$ if and only if the induced maps $f^0:H^0(Y;S) \to H^0(X;S)$ for every profinite set $S$, $f^1:H_{\mathrm{cts}}^1(Y;\Gamma) \to H_{\mathrm{cts}}^1(X;\Gamma)$ for every every profinite group $\Gamma$ and $f^{\ast}:H_{\mathrm{cts}}^q(Y;\Mh) \to H_{\mathrm{cts}}^q(X;f^{\ast}\Mh)$ is an isomorphism for every continuous local coefficient system $\Mh$ of profinite abelian groups on $Y$ for every $q\geq 0$.
\end{prop}
\begin{proof}
From $H^0(X;S)= \Hom_{\hEh}(\pi_0(X),S)$ for every finite set $S$, we conclude that $\pi_0(f)$ is an isomorphism if and only if $H^0(f;S)$ is an isomorphism for every profinite set $S$. So we can assume $X$ and $Y$ are connected. From \cite{profinhom}, Lemma 2.9, we get that $\pi_1(f)$ is an isomorphism if and only if $H_{\mathrm{cts}}^1(f;\Gamma)$ is an isomorphism for every finite group $\Gamma$. Hence the if-part of the assertion is proved. It remains to show that the statement extends from finite to profinite coefficients. This can be shown using a spectral sequence that relates local cohomology with finite and with profinite coefficients by noting that continuous cohomology can be expressed in the followingway by homotopy groups of mapping spaces. For any profinite module $\Mh$ over $\Pi=\Pi_1X$, there is an isomorphism 
$$H_{\mathrm{cts}}^{n-q}(X;\Mh)=\pi_q\hom_{\hSh/B\Pi}(X,K(\Mh,n)),$$
where $\pi_q$ denotes the usual homotopy group of the space $\hom_{\hSh/B\Pi}(X,K(\Mh,n))$ of maps in $\hSh$ over $B\Pi$. For an arbitrary profinite group $\Gamma$ there is a bijection of pointed sets 
$$H^{1}(X;\Gamma)=\pi_0\hom_{\hShp}(X,B\Gamma).$$ 
Then we can construct a Bousfield-Kan spectral sequence as in \cite{dwyfried}, Proposition 2.9., that yields the comparison of coefficients.
\end{proof}

\begin{prop}\label{sgdfib}
Let $\Gamma$ be a simplicial profinite groupoid. The profinite classifying space $\bar{W}\Gamma$ is fibrant in $\hSh$.
\end{prop}
\begin{proof}
The proof is basically the same as for a simplicial profinite group. 
We start with a simplicial finite groupoid $\Gamma$ and let $\Pi_1$ be the fundamental finite groupoid of $\bar{W}\Gamma$. We denote by $\Pi_n$ the finite $\Pi_1$-module of $\bar{W}\Gamma$ defined as in the example above. We know that the underlying simplicial set of $\bar{W}\Gamma$ is a fibrant object in $\Sh$ such that, for every object $x$, the profinite space $\bar{W}\Gamma(x,x)$ is minimal fibrant in $\hSh$. Together with the theory of Postnikov towers this implies that there is a pullback square of simplicial finite sets
\begin{equation}\label{coskpbgd}
\xymatrix{
\cosk_{n+1}\bar{W}\Gamma \ar[d] \ar[r] & \hocolim_{\Pi_1} L(\Pi_n,n+1) \ar[d]^q \\
\cosk_n \bar{W}\Gamma \ar[r]_{k_n} & \hocolim_{\Pi_1} K(\Pi_n,n+1).}
\end{equation}
The $k$-invariant $k_n$ is given as a relative cocycle in $Z^{n+1}(\cosk_n\bar{W}\Gamma; \Pi_n)$ over $B\Pi_1$ as the map that assigns to every $n+1$-simplex $w: \sk_n\Delta^{n+1}\to \bar{W}\Gamma$ the corresponding element in $\Pi_n(w(0))=\pi_n(\bar{W}\Gamma,w(0))$.
It follows from the definition of weak equivalences and the lifting properties that the map $q$ on the right is contained in the saturation of the generating set $P$ of fibrations in $\hSh$. Hence the map $\cosk_{n+1}\bar{W}\Gamma \to \cosk_n \bar{W}\Gamma$ is a fibration in $\hSh$ for $n\geq 2$. Moreover, $\cosk_2\bar{W}\Gamma$ is $B\Pi_1$ which is fibrant in $\hSh$. Hence $\bar{W}\Gamma$ is a fibrant profinite space.\\ 
%
%
For a simplicial profinite groupoid $\Gamma$, we observe again that all diagrams and objects involved commute with the profinite structure of $\Gamma$. Hence we obtain pullback diagrams in $\hSh$. Since the map $q$ in diagram (\ref{coskpbgd}) is contained in the saturation of the generating fibrations of $\hSh$ for profinite $\Pi_1$ and $\Pi_n$ by Proposition \ref{wepf}, we can conclude that all maps in the tower of coskeleta of $\bar{W}\Gamma$ are fibrations in $\hSh$. This finishes the proof.
\end{proof}
\begin{remark}
Note that the arguments above do not, of course, show that every profinite space, whose underlying simplicial set is fibrant in $\Sh$, is also fibrant as an object in $\hSh$. In the proofs of the two propositions we have used very special properties of the profinite space $\bar{W}\Gamma$. In particular, we used the minimality of the functor $\bar{W}$ and that the homotopy groups of $\bar{W}\Gamma$ (or rather its underlying simplicial set) commute with the limit structure of $\Gamma$ as a simplicial profinite group or groupoid. As we pointed out at numerous places, the last property is not satisfied by a general profinite space.
\end{remark}
\subsection{A fibrant replacement functor in $\hSh$}\label{secfibrep}
As indicated by Morel for pro-$p$-completion of spaces in \cite{ensprofin}, \S 2.1, p. 367, the constructions above yield an explicit fibrant replacement functor in $\hSh$. This idea is based on the work of Quillen in \cite{quillen2} and of Rector in \cite{rector}.\\
First, let $X$ be a reduced simplicial finite set and let $\hat{\Gamma}X$ be its free simplicial profinite loop group. Its profinite classifying space $\bar{W}\hat{\Gamma}X$ is a fibrant profinite space by Proposition \ref{sgfib} and is equipped with the canonical map $\heta: X \to \bar{W}\hat{\Gamma}X$.\\
Second, let $X$ be an arbitrary simplicial finite set and let $\hat{\Gamma}X$ be its free simplicial profinite loop groupoid. The profinite classifying space $\bar{W}\hat{\Gamma}X$ is fibrant in $\hSh$ by Proposition \ref{sgdfib} and is equipped with the canonical map $\heta: X \to \bar{W}\hat{\Gamma}X$ in $\hSh$. If $X$ is an arbitrary profinite space, it is isomorphic in $\hSh$ to the limit $\lim_R X/R$ where $R$ runs through the simplicial open equivalence relations on $X$. Its free simplicial profinite groupoid is $\hat{\Gamma}X = \lim_R \hat{\Gamma}(X/R)$.
Then we apply $\bar{W}$ to get a fibrant profinite space $R_fX$ equipped with a canonical map in $\hSh$ 
$$\heta: X\to R_fX:=\bar{W}\hat{\Gamma}X=\lim_R \bar{W}\hat{\Gamma}(X/R).$$
Since the construction of $\heta$ is natural in $X$, the following theorem justifies to call $R_fX$ a functorial fibrant replacement of $X$. 
\begin{theorem}\label{fibrep}
Let $X$ be a profinite space. The map  $\heta: X \to \bar{W}\hat{\Gamma}X$ is a trivial cofibration in $\hSh$ and $\bar{W}\hat{\Gamma}X$ is a fibrant profinite space.
\end{theorem}
Before we start the proof, we need the following fact about simplicial groupoids due to Goerss and Jardine, cf. \cite{gj}, V \S 7, pp. 316-317. Let $\Gamma$ be a simplicial groupoid. Picking a representative $x\in [x]$ for each $[x] \in \pi_0\Gamma$, defines a map of simplicial groupoids
$$i: \bigsqcup_{[x]\in \pi_0\Gamma} \Gamma(x,x) \to \Gamma$$
from the disjoint union of simplicial groups $\Gamma(x,x)$ to $\Gamma$. This map is not only a weak equivalence of simplicial groupoids but $\bigsqcup_{[x]\in \pi_0\Gamma} \Gamma(x,x)$ is a deformation retract of $\Gamma$. In particular, the map $i$ is a homotopy equivalence of simplicial groupoids, see \cite{gj}, V \S 7, pp. 316-317, for a definition of a groupoid homotopy. We will need the following consequence for the completion of a groupoid.
\begin{lemma}\label{sghe}
Let $\Gamma$ be a simplicial groupoid with a finite set of objects and let $\hat{\Gamma}$ be its profinite completion. Then, for any choice of representatives $x\in [x]$ for $[x] \in \pi_0\Gamma=\pi_0\hat{\Gamma}$, the induced map 
$$i: \bigsqcup_{[x]\in \pi_0\hat{\Gamma}} \hat{\Gamma}(x,x) \to \hat{\Gamma}$$
is still a homotopy equivalence of simplicial groupoids.
\end{lemma}
Now we can start the proof of Theorem \ref{fibrep}.
\begin{proof}
It is clear that $\heta$ is a cofibration, i.e.\,a monomorphism in each level. That $\bar{W}\hat{\Gamma}X$ is a fibrant profinite space follows from Proposition \ref{sgdfib}.
In order to show that it is a weak equivalence in $\hSh$ we can assume that $X$ is a simplicial finite set. For, $\heta$ is defined as the limit of the maps $\heta/R:X/R \to \bar{W}\hat{\Gamma}(X/R)$. If we show that each $\heta/R$ is a weak equivalence, then the homotopy invariance of limits in $\hSh$, Proposition 2.14 in \cite{profinhom}, implies that $\heta$ is a weak equivalence in $\hSh$ as well.\\ 
So let $X$ be a simplicial finite set. 
The set of objects $\mathrm{Ob}(\Gamma X)$ is equal to the set of vertices $X_0$ of $X$. Since this set is finite, we also have $\mathrm{Ob}(\hat{\Gamma} X)=X_0$ and $\pi_0\hat{\Gamma}X=\pi_0X$. Since the functor $\bar{W}$ preserves disjoint unions, there is a commutative diagram in $\hSh$
$$\xymatrix{
\bigsqcup_{[x]\in \pi_0X} X_x \ar[d] \ar[r] & \bigsqcup_{[x]\in \pi_0\hat{\Gamma}X} \bar{W}(\hat{\Gamma}X(\heta(x), \heta(x))) \ar[d]^{\bar{W}(i)} \\
X \ar[r] & \bar{W}\hat{\Gamma}X}$$
where $X_x$ denotes the connected component of $X$ corresponding to $x$. By Lemma \ref{sghe}, the vertical map $i$ on the right hand side is a homotopy equivalence of simplicial groupoids. As explained in \cite{gj}, V Proof of Theorem 7.8, this implies that $\bar{W}(i)$ is a homotopy equivalence of simplicial sets. Hence $\bar{W}(i)$ is also a weak equivalence in $\hSh$ by invariance of fundamental groups and cohomology under homotopy. Since the left vertical map is a weak equivalence, we conclude that it suffices to prove the assertion for each vertex of $X$ separately. Thus we can assume that $X$ is a reduced simplicial finite set.\\
We know from \cite{gj}, Proposition 6.3, that $\eta: X \to \bar{W}\Gamma X$ is a weak equivalence in $\Sh$. To be able to deduce from this a statement about $\heta$ we have to take into account the effect of profinite completion.\\
For this proof only, we will use the notation $\hat{\pi}_1X$ to denote the profinite fundamental group of $X$ considered as an object in $\hSh$ to distinguish it from its fundamental group as an object in $\Sh$. 
We know that $\eta$ induces an isomorphism of fundamental groups of simplicial sets $\pi_1X \cong \pi_1 \bar{W}\Gamma X = \pi_0 (\Gamma X)$. 
The profinite fundamental group $\hat{\pi}_1X$ of $X$ is the profinite completion of the fundamental group of $X$ as an object in $\Sh$ by Proposition \ref{profinitecompletion}. Similarly, the profinite fundamental group of $\bar{W}\hat{\Gamma}X$ is just the completion of $\pi_1\bar{W}\Gamma X$, since $\pi_0$ commutes with filtered inverse limits of simplicial finite groups and 
$$\hat{\pi}_1 \bar{W}\hat{\Gamma}X= \pi_0\hat{\Gamma}X = \lim_U \pi_0 (\Gamma X)/U = \lim_U (\pi_1 X)/U.$$
Hence $\heta$ induces an isomorphism on profinite fundmental groups.\\
It remains to show that $\heta$ induces an isomorphism on cohomology with finite local coefficients. Since $X$ is reduced, finite local coefficient systems on $X$ are just finite discrete $\hat{\pi}_1X$-modules. 
If $\hat{\Gamma}$ a simplicial profinite group and $M$ is a continuous discrete $\pi_0\hat{\Gamma}$-module, the continuous cohomology of $\hat{\Gamma}$ with coefficients in $M$ is given by
$$H_{\mathrm{cts}}^*(\hat{\Gamma}; M) = \colim_U H^*(\bar{W}(\hat{\Gamma}/U); M^{\pi_0U}),$$
where the limit is taken over all open normal subgroups of $\hat{\Gamma}$ and $M^{\pi_0U}$ denotes the module of fixed elements under $\pi_0U$. Furthermore, if $\Gamma$ is a simplicial group, the profinite group completion map $\Gamma \to \hat{\Gamma}$ induces a canonical map 
$$H_{\mathrm{cts}}^*(\hat{\Gamma};M) \longrightarrow H^*(\Gamma;M)$$
for every finite discrete $\pi_0\hat{\Gamma}$-module. As for groups, this map is not an isomorphism in general. Quillen calls the simplicial group $\Gamma$ \emph{good}, if this map is an isomorphism for every finite discrete $\pi_0\hat{\Gamma}$-module, see \cite{quillen2}. Using a spectral sequence argument, one can show that a simplicial group $\Gamma$ is good if $\Gamma_n$ is a good group in the sense of Serre for all $n$.\\
Coming back to the proof of Theorem \ref{fibrep}, the crucial observation is that free groups are good, cf. \cite{quillen2}, Proposition 3.1. Since $\Gamma X_n$ is by definition a free group in each degree, we conclude that $\Gamma X$ is a good simplicial group. Thus we get an isomorphism
$$H_{\mathrm{cts}}^*(\hat{\Gamma}X;M) \stackrel{\cong}{\longrightarrow} H^*(\Gamma X;M)$$
for all continuous finite $\pi_0\Gamma X=\hat{\pi}_1X$-module $M$. 
We have seen above that the map $X\to \bar{W}\Gamma X$ is a weak equivalence of simplicial sets and hence $H^*(\Gamma X;M)=H^*(\bar{W}\Gamma X;M)=H^*(X;M)$. Finally, when $M$ is finite and $X$ is a simplicial finite set, the continuous cohomology $H_{\mathrm{cts}}^*(X;M)$ of $X$ agrees with the cohomology $H^*(X;M)$.
This completes the proof that $X\to \bar{W}\hat{\Gamma}X$ is a weak equivalence of profinite spaces.
\end{proof}
%
%
\subsection{Relationship to the work of Artin-Mazur, Morel and Sullivan}\label{secrelation}
In the previous subsection we have constructed a functorial fibrant replacement in $\hSh$. For a simplicial finite set $X\in \hSh$ it is given as the map $\eta: X\to \bar{W}\hat{\Gamma}X$. 
Now $\hat{\Gamma}X$ is by definition given as the (simplicial) profinite groupoid completion of the free simplicial groupoid $\Gamma X$. Hence $\hat{\Gamma}X$ is by definition a limit of simplicial finite groupoids. By taking the classifying space functor $\bar{W}$ we get a decomposition $\bar{W}\hat{\Gamma}X$ as a limit of simplicial finite sets fibrant in $\hSh$
$$\bar{W}\hat{\Gamma}X = \lim_U \bar{W}((\Gamma X)/U)$$
where $U$ runs through the simplicial normal subgroupoids of $\Gamma X$ such that the quotient $(\Gamma X)/U$ is a simplicial finite groupoid. As we have seen before, each of the $\bar{W}((\Gamma X)/U)$ is fibrant in $\hSh$ and hence in $\Sh$. Moreover, the homotopy groups of each $\bar{W}(\Gamma X)/U$ are finite by Lemma \ref{finitehomgroups} below. So after taking Postnikov sections $\cosk_n \bar{W}(\Gamma X)/U$ we get a decomposition into a limit of finite spaces which are also simplicial finite sets, i.e. a weak equivalence in $\hSh$
$$X \stackrel{\simeq}{\longrightarrow} \lim_{n,U} \cosk_n (\bar{W}((\Gamma X)/U)).$$
\begin{lemma}\label{finitehomgroups}
Let $X$ be a fibrant simplicial finite set. Then its homotopy groups $\pi_n(X,x)$ are finite groups for every $n\geq 0$ and every vertex $x$.
\end{lemma}
\begin{proof}
The $n$th homotopy group $\pi_n(X,x)$ of a fibrant simplicial set $X$ is defined to be the set of homotopy classes of maps $\alpha:\Delta^n \to X$ (relative $\partial \Delta^n$) which fit into diagrams 
$$\xymatrix{
\Delta^n \ar[r]^{\alpha} & X \\
\partial \Delta^n \ar[u] & \ar[l] \Delta^0 \ar[u]_{x}.}
$$
Since $X$ is fibrant and $\Delta^n$ cofibrant, the set of homotopy classes of maps is just the quotient $\Hom_{\Sh}(\Delta^n,X)/\sim$ of maps modulo the simplicial homotopy relation. But since  the set $X_n$ of $n$-simplices of $X$ is finite by our assumption and $\Hom_{\Sh}(\Delta^n,X)=X_n$, there are only finitely many maps and $\pi_n(X,x)$ is a finite group.
\end{proof}

Now let $X$ be an arbitrary profinite space. Then the previous construction yields a decomposition of $X$ as a limit of finite spaces
$$X \stackrel{\simeq}{\longrightarrow} \hat{X}_f:=\lim_{n,R,U} \cosk_n (\bar{W}((\Gamma(X/R)/U).$$
Since the groups $\pi_k \bar{W}(\Gamma(X/R)/U)$ are finite, their higher lim-terms vanish and a spectral sequence argument using sequence (3) and Lemma 2.18 of \cite{profinhom} shows for every $k\geq 0$
$$\pi_kX=\lim_{n,R,U} \pi_k(\cosk_n \bar{W}(\Gamma(X/R)/U))$$
where $\pi_kX$ on  the left denotes the profinite homotopy group of the profinite space $X$.\\
In particular, we have constructed a functor from $\hSh$ to the category of pro-objects of finite spaces defined by
$$X \mapsto \{\cosk_n\bar{W}(\Gamma(X/R)/U)\}_{n,R,U}.$$
By applying this to the set-theoretic completion $\hat{X}$ of a space $X$, we get a functor 
$$F:\Sh \to \hSh \to \mathrm{pro}-\Sh_{\mathrm{fin}}$$ 
where $\Sh_{\mathrm{fin}}$ is the subcategory of $\Sh$ of finite spaces. 
We can consider this functor on the homotopy level 
$$\Hh \to \hHh \to \mathrm{pro}-\Hh_{\mathrm{fin}}.$$ 
It follows immediately from the results on profinite spaces and \cite{artinmazur}, Theorem 4.3, that $F$ is isomorphic to the Artin-Mazur completion functor. Moreover, this implies that the fibrant replacement of $\hat{X}$ in $\hSh$ is a rigid model for the Sullivan completion of $X$, i.e. that $|\hat{X}_f|$ is isomorphic to $\hat{X}^{\mathrm{Su}}$ in $\Hh$. Hence $X\mapsto \hat{X}_f= \lim F(X)$ provides a rigid model for the profinite completion of a space $X$.
%
%
%
\subsection{Completion of spaces versus completion of groups II}\label{sullivan}
We return to the question how completion of spaces and groups are related to each other. We have seen that this a subtle problem. It turns out that after modifying slightly the notion of good groups for higher homotopy groups, one gets a sufficient condition such that the completion of spaces commutes with the one of groups. This result is due to Sullivan. We state it in our terminology to complete the picture for the reader.\\ 
Following \cite{sullivan}, for a pointed space $X$, we call $\pi_1:=\pi_1X$ a {\em good fundamental group}, if it is a good group and has finite cohomology groups, i.e. if the map $H_{\mathrm{cts}}^i(\hat{\pi}_1;M) \to H^i(\pi_1;M)$ is an isomorphism and if these groups are finite for all finite $\pi_1$-modules $M$ and all $i\geq0$.\\
Let $\pi_n:=\pi_nX$, $n \geq 2$, be a higher homotopy group of $X$. It carries a canonical action of $\pi_1$. Let $\Ph$ be the filtered set of finite $\pi_1$-quotients of $\pi_n$. We denote by $\hat{\pi}^{\pi_1}_n:=\lim_{Q\in\Ph}\pi_n/Q$ the $\pi_1$-completion of $\pi_n$. This is, in particular, a profinite group on which $\pi_1$ acts. The $\pi_1$-module $\pi_n$ is called a {\em good higher homotopy group} if
$$H_{\mathrm{cts}}^i(\hat{\pi}^{\pi_1}_n;A) \cong H^i(\pi_n;A)$$ 
and if these groups are finite for all finite coefficient groups $A$ and all $i\geq0$. With these definitions there is the following result of Sullivan \cite{sullivan}, Theorem 3.1. It holds for our rigid model $\hat{X}_f$ of the profinite completion of $X$, since $\hat{X}_f$ is isomorphic to $\hat{X}^{\mathrm{Su}}$ in $\Hh$. 
\begin{theorem}
Let $X$ be a connected pointed space. If $X$ has a good fundamental group and good higher homotopy groups, then the canonical map $\varphi_t: \widehat{\pi_tX} \to \pi_t\hat{X}$ is an isomorphism of profinite groups for every $t$.
\end{theorem}
Sullivan shows that groups which are commensurable with solvable groups in which every subgroup is finitely generated are good fundamental groups and that finitely generated abelian groups are good higher homotopy groups. In particular, we have the following immediate consequence of the previous theorem.
\begin{cor}\label{finitecompletion}
Let $X$ be a space whose homotopy groups are all finite. Then profinite completion induces an isomorphism $\pi_tX=\pi_t\hat{X}$ for every $t\geq 0$.
\end{cor}
\subsection{Completion of nilpotent spaces}
There is another condition for the homotopy groups of $X$ that allows to get our hands on the relation between $\pi_nX$ and $\pi_n\hat{X}$. Therefore let $X$ be a nilpotent space. This means that $\pi_1X$ is a nilpotent group and the action of $\pi_1X$ on the abelian groups $\pi_nX$ for $n\geq 2$ is also nilpotent, i.e. $\pi_nX$ has a finite filtration such that $\pi_1X$ acts trivially on each quotient of the filtration.\\
So let $X$ be a connected nilpotent space and let $p$ be a prime number. Paul Goerss has shown in \cite{goersscomp}, Proposition 5.9, that the homotopy groups of the pro-$p$-completion $\hat{X}_p$ of $X$, i.e. the fibrant replacement of $\hat{X}$ in the $\Z/p$-model structure on $\hSh$ of Morel \cite{ensprofin}, fit in a splittable short exact sequence
$$0 \to \widehat{\pi_nX}^p \to \pi_n\hat{X}_p \to \hat{L}^p_1(\pi_{n-1}X)\to 0$$
for every $n\geq 1$. Here $\hat{L}^p_1$ denotes the first left derived functor of the pro-$p$-group completion functor, which can be defined for non-abelian groups as in \cite{goersscomp}, Definition 5.6, via Eilenberg-MacLane spaces.\\
This result has independently been proven for the full profinite completion in terms of pro-spaces by Rector in \cite{rector}, Theorem 5.11, for a slightly restricted class of connected nilpotent spaces. In loc.\,cit., Rector shows that there is a natural short exact sequence for every $n\geq 1$
$$0 \to \widehat{\pi_nX} \to \pi_n\hat{X}_f \to \hat{L}_1(\pi_{n-1}X)\to 0$$
where $\hat{L}_1$ denotes the first left derived functor of the pro-p-group completion functor.\\
Another nice property of a connected nilpotent space $X$ is that its profinite homotopy type is determined by its pro-$p$-types. For any space $X$ there is a canonical natural map 
$$\hat{X}_f \to \prod_p \hat{X}_p.$$
This map is an equivalence if $X$ is nilpotent and of finite type by \cite{sullivan}, p. 53. 
%
%
\section{Profinite $G$-spaces}
We turn our attention to the equivariant setting. Let $G$ be a fixed profinite group and let $S$ be a profinite set on which $G$ acts continuously, i.e.\,there is a continuous map $\mu:G \times S \to S$ satisfying $\mu(e,s)=s$ and $\mu(gh,s)=\mu(g,\mu(h,s))$ for all $s\in S$, $g,h \in G$ and $e\in G$ being the neutral element. In this situation we say that $S$ is a profinite $G$-set. If $X$ is a profinite space and $G$ acts continuously on each $X_n$ such that the action is compatible with the structure maps, then we call $X$ a {\em profinite $G$-space}. We denote by $\hShg$ the category of profinite $G$-spaces with $G$-equivariant maps of profinite spaces as morphisms.\\ 
While a discrete $G$-space $Y$ is characterized as the colimit over the fixed point spaces $Y^U$ over all open subgroups, a profinite $G$-space X is the limit over its orbit spaces $X/U$. More explicitly, for an open and hence closed normal subgroup $U$ of $G$, let $X/U$ be the quotient space under the action by $U$, i.e.\,the quotient $X/\sim$ with $x \sim y$ in $X$ if both are in the same orbit under $U$. 
\begin{lemma}\label{actionlemma}
Let $G$ be a profinite group and $X$ a profinite space with a $G$-action. Then $X$ is a profinite $G$-space if and only if the canonical map $\phi:X \to \lim_U X/U$ is an isomorphism, where $U$ runs through the open normal subgroups of $G$. 
\end{lemma}
\begin{proof}
It suffices to prove this for each $X_n$, so let $X$ be a profinite $G$-set. The equivalence relation on $X$ induced by the action of $U$ is an open and closed relation, see e.\,g.\,\cite{bourb} Chapter I-III for the topological results we use. Hence the quotients $X/U$ are again Hausdorff spaces. Since the cofiltered limit of compact Hausdorff spaces is so again, we deduce that the limit $\lim_U X/U$ is a totally disconnected compact Hausdorff space. Now each map $X \to X/U$ is surjective and hence the image of $\phi:X\to \lim_U X/U$ is dense. Since $X$ is compact and $\lim_U X/U$ is a compact Hausdorff space, $\phi(X)$ is already closed and $\phi$ is an open and surjective map. For the injectivity, let $x\neq y$ be two distinct points in $X$. Since $X$ is Hausdorff, there is an open subset $V$ of $X$ that contains $x$ but does not contain $y$. The preimage $\mu_x^{-1}(V)$ of $V$ under the continuous map $\mu_x: G \to X,~g \mapsto \mu(g,x)$ is an open subset of $X$. Now $G$ being a profinite group, the open normal subgroups of $G$ form a basis of the topology on $G$. Hence $\mu_x^{-1}(V)$ contains at least one open normal subgroup $U$. Then $y$ is not in the orbit $Ux$ of $x$ under $U$. Hence $\phi(x)\neq \phi(y)$ in $\lim_U X/U$. This shows that $\phi$ is a continuous bijection between compact Hausdorff spaces and hence $\phi$ is a homeomorphism. 
\end{proof}

Moreover, every profinite $G$-set is in fact a limit of finite $G$-quotients by \cite{ribes}, Lemma 5.6.4. This yields an analogue decomposition of a profinite $G$-space.
\begin{lemma}\label{5.6.4}
Let $G$ be a profinite group and $X$ a profinite $G$-space. There is $G$-invariant decomposition of $X$ as an inverse limit of simplicial finite $G$-sets 
$$X =\lim_i X_i~\mathrm{in}~\hShg.$$
\end{lemma}
\begin{proof}
Again, it suffices to prove the assertion for a profinite $G$-set $X$. This is done in Lemma 5.6.4 of \cite{ribes}. The idea is to show that for any open equivalence relation $R$ on $X$, considered as an open subset of $X \times X$, there is a $G$-invariant open equivalence relation $S\subseteq R$. One defines $S$ to be the intersection of all open subsets $gR$ in $X$, i.e.
$$S=\bigcap_{g\in G} gR.$$
Using that $G$ and $X$ carry a profinite topology, one can show that $S$ is in fact an open subset of $X\times X$. This implies that each quotient $X/S$ is a finite $G$-set. Finally, one shows that $X$ is equal to the limit $\lim_S X/S$ as in the proof of the previous lemma.
\end{proof}

In order to get a model structure on $\hShg$ one can find explicit sets of generating fibrations and trivial fibrations. They arise naturally by considering $G$-actions on the corresponding generating sets for the model structure on $\hSh$. The following result has been proven in \cite{gspaces}, Theorem 2.9.
\begin{theorem}\label{Gmodelunstable}
There is a fibrantly generated left proper simplicial model structure on the category of profinite $G$-spaces such that a map $f$ is a weak equivalence (respectively fibration) in $\hShg$ if and only if its underlying map is a weak equivalence (respectively fibration) in $\hSh$. A map $f:X\to Y$ is a cofibration in $\hShg$ if and only if $f$ is a level-wise injection and the action of $G$ on $Y_n -f(X_n)$ is free for each $n \geq 0$.
We denote its homotopy category by $\hHhg$. 
\end{theorem}
%
%
\subsection{Equivariant completion}
Let $G$ be a profinite group and $X$ a simplicial $G$-set, i.e. a simplicial object in the category of $G$-sets (without any topological condition). We are going to construct a functorial completion such that the output is a profinite space with a continuous $G$-action, i.e. an object of $\hShg$.\\ 
Therefore, we start with a finite group $F$ and a simplicial $F$-set $Y$. Let $\Rh_F(Y)$ be the set of $F$-equivariant equivalence relations on $Y$ with finite quotients, i.e. each $R\in \Rh_{F}$ is an $F$-invariant simplicial subset of $Y\times Y$ such that $(Y_n\times Y_n)/R_n$ is finite. The limit $\hat{X}_F:=\lim_{R\in \Rh_{F}} Y/R$ is an $F$-invariant profinite completion.\\
Now let $G$ be again our profinite group and $X$ a simplicial $G$-set. If we defined $\hat{X}_G$ in the same way as for $F$, the induced $G$-action on $\hat{X}_G$ would in general not be continuous. Instead, we first consider the quotient $X/U$ by an open normal subgroup $U$ of $G$. Via the just defined $G/U$-equivariant completion we obtain a profinite $G/U$-space $\widehat{(X/U)}_{G/U}$. The limit over all open normal $U$ is a profinite space with a continuous $G$-action. 
\begin{defn}\label{defGcomp}
We define the profinite $G$-completion of $X$ to be 
$$\hat{X}_G:=\lim_U \widehat{(X/U)}_{G/U}$$
where $U$ runs through the open normal subgroups of $G$.
\end{defn}
It is equipped with a $G$-equivariant map $\varphi:X\to \hat{X}_G$ and has the following expected universal property.
\begin{lemma}
Let $Z$ be a profinite $G$-space and let $f:X\to Z$ be a map of simplicial $G$-sets. Then there is a unique map $\hat{f}$ such that $f$ factors as $f= \hat{f}\circ \varphi$.
\end{lemma}
\begin{proof}
Since $Z$ is a profinite $G$-space it is isomorphic to $\lim_U Z/U$ by Lemma \ref{actionlemma}. Hence, by definition of $\hat{X}_G$, we can assume that $G$ is finite. Then $\hat{X}_G$ is equal to $\lim_{R\in \Rh_{G}} X/R$ and it becomes obvious that $f$ factors uniquely through $\varphi$.
\end{proof}
\begin{remark}
Boggi \cite{boggi}, \cite{boggi2} and Lochak \cite{lochak} study also a $G$-completion functor for $G$-spaces. But in loc.\,cit.\,one starts with a discrete group $G$ acting on a simplicial set $X$ satisfying certain additional conditions. One obtains a profinite $\hat{G}$-space in our sense where $\hat{G}$ is the profinite group completion of $G$ (or a quotient of the full profinite completion). Here we start with a profinite group. The two approaches converge to a common result if $G$ is a strongly complete profinite group. We will discuss this property below.
\end{remark}
\subsection{Discrete $G$-spaces}
Let $G$ be a profinite group. Defining $G$-equivariant completion is more convenient when we require a natural assumption for the action of $G$ on the spaces. So we restrict our attention to discrete $G$-spaces in the sense of \cite{goerss}. A simplicial $G$-set is called a discrete $G$-space when the action of $G$ on each set $X_n$ equipped with the discrete topology is continuous and compatible with the face and degeneracy maps. Let $\Sh_{dG}$ be the category of such discrete $G$-spaces with $G$-equivariant maps as morphisms. 
We can simplify the profinite completion functor, when we restrict it to $\Sh_{dG}$. We denote again by $\Rh_{G}(X)$ the set of $G$-invariant simplicial equivalence relations on $X$ such that the quotient $X/R$ is a simplicial finite set. Then we get the following simplification.
\begin{prop}\label{compdiscrete}
Let $X$ be a discrete $G$-space. Then the limit 
$$\hat{X}'_G:=\lim_{R\in \Rh_{G}(X)} X/R$$ 
is a profinite $G$-space and is isomorphic to the $G$-completion $\hat{X}_G$ of $X$ of Definition \ref{defGcomp}.
\end{prop}
\begin{proof}
Since the $G$-action is compatible with the simplicial structure, it suffices again to prove the corresponding assertion for discrete and profinite $G$-sets. So let $X$ be a discrete $G$-set. 
Every $R\in \Rh_{G}$ is an open equivalence relation. Hence the continuous action of $G$ on $X$ induces a continuous action of $G$ on the finite discrete sets $X/R$. Thus $(\hat{X}'_G)$ is the limit of continuous $G$-sets and hence it is itself a continuous $G$-set. So $\hat{X}'_G$ is a profinite $G$-space. 
Moreover, it follows immediately from the construction that $\hat{X}'_G$ satisfies the same universal property for $G$-equivariant maps from the discrete $G$-space $X$ to profinite $G$-spaces as $\hat{X}_G$. Thus there is a unique isomorphism $\hat{X}'_G \cong \hat{X}_G$ in $\hShg$. 
\end{proof}
\begin{remark}
There is a model structure on $\Sh_{dG}$ constructed by P. Goerss in \cite{goerss}, Theorem 1.12, in which a map is a weak equivalence (resp.\,cofibration) in $\Sh_{dG}$ if and only if it is a weak equivalence (resp.\,cofibration) in $\Sh$. In particular, every discrete $G$-space is cofibrant.
The relationship between $G$-equivariant completion functor and the functor that forgets the profinite structure on a profinite $G$-space is not as nice as in the non-equivariant case. The problem is that the underlying set of a profinite $G$-space is not a discrete $G$-space. Moreover, a profinite $G$-space $X$ is cofibrant if and only of the action is free on each $X_n$ as in the case for $\Sh_G$. Since all discrete $G$-spaces are cofibrant in the model structure of \cite{goerss}, cofibrations are not preserved by $G$-equivariant completion. So an analogue of Proposition \ref{adjcompletion} cannot be formulated for discrete and profinite $G$-spaces. 
\end{remark}
\subsection{Strongly complete profinite groups}
Now we return to abritrary $G$-spaces, but we require that $G$ has an additional property. A profinite group $G$ is called strongly complete in \cite{ribes}, if every subgroup of finite index is also open in $G$, or, in other words, if $G= \hat{G}$ as profinite groups. The profinite completion of an abstract group is itself strongly complete. But in general there are subgroups of finite index which are not open in the given topology. A discussion of this phenomenon is given in \cite{ribes} \S 4.2.\\
Serre has conjectured that every topologically finitely generated profinite group $G$ is strongly complete, where topologically finitely generated means that $G$ contains a dense finitely generated subgroup. He proved this conjecture for finitely generated pro-$p$-groups. Recently, Nikolov and Segal have proven the full conjecture for every finitely generated profinite group in \cite{niksegal}. The celebrated proof relies on the classification of finite simple groups.\\
The implication of strong completeness, which shows why it is interesting in this context, is the following. For a strongly complete profinite group $G$, every finite set $S$ with a $G$-action is also a continuous  discrete $G$-set. For, if $s$ is an element in $S$ and $Gs$ the orbit of $s$ under $G$ in $S$, then $G$ acts transitively on $Gs$ and hence there is bijection $G/G_s\cong Gs$, where $G_s$ denotes the stabilizer of $s$ in $G$. Since $Gs \subseteq S$ is a finite set, so is $G/G_s$, and, since $G$ is strongly complete, $G_s$ is open in $G$.\\
(One should note however that this does not imply that a strongly complete group is good in the sense of Serre. This is because the profinite completion of $G$ as a group is in general not equal to the profinite completion as a set. Hence the sets $\Hom_{\Eh}(G,S)$ and $\Hom_{\hEh}(\hat{G},S)$ can still be different for a finite set $S$.)\\
Hence when $G$ is strongly complete, every simplicial finite $G$-set is a simplicial discrete $G$-set. So profinite completion of any $G$-space $X$ can be defined more directly as in the case of discrete spaces. For any simplicial open $G$-invariant equivalence relation $R$ on $X$ with finite quotients, the finite set $X_n/R_n$ is finite and a continuous $G$-set for the discrete topology. The proof of Proposition \ref{compdiscrete} above shows that $\hat{X}_G$ is equal to the limit $\lim_{R\in \Rh_{G}(X)} X/R$, where $\Rh_{G}(X)$ denotes the set of $G$-invariant simplicial equivalence relations on $X$ such that the quotient $X/R$ is a simplicial finite set.\\
Arithmetically interesting examples of strongly complete groups are the absolute Galois groups of finite fields. More subtle examples provide the Galois groups of $p$-adic local fields, since they are finitely generated by the work of Jannsen and Wingberg \cite{jw}. Nevertheless, the absolute Galois group of a number field is in general not strongly complete as subgroups of finite index which are not open can be constructed in such groups. Another interesting example is provided by the  Morava stabilizer group of formal group laws in characteristic $p$ by the work of Ravenel \cite{ravenel}.\\ 
In order to understand the fundamental group of the $G$-completion of a space we recall the equivariant completion for groups. 
Let $\Gamma$ be a group that is equipped with a left $G$-action $\mu:G\times \Gamma \to \Gamma$ such that $\mu(g, \gamma_1 \gamma_2)=\mu(g,\gamma_1) \mu(g,\gamma_2)$ in $\Gamma$ for every $g\in G$ and $\gamma_1, \gamma_2 \in \Gamma$. We define the $G$-equivariant group completion of $\Gamma$ to be the limit of finite groups
$$\hat{\Gamma}_G:= \lim_{U_G} \Gamma/U_G$$
where $U_G$ runs through the normal subgroups of $\Gamma$ of finite index which are invariant under the action of $G$. It is clear that $\hat{\Gamma}_G$ is a profinite group with a $G$-action in the above sense. Since we restricted to the case that $G$ is strongly complete, each $\Gamma/U_G$ is a finite and hence discrete $G$-group. Thus $\hat{\Gamma}_G$ is a profinite group with a  continuous $G$-action and the map $\Gamma \to \hat{\Gamma}_G$ is universal for maps from $\Gamma$ into groups with this property. An example for this construction is given by the profinite fundamental group of a $G$-space.
\begin{prop}\label{Gfundamental}
Let $G$ be a strongly complete profinite group and $X$ a simplicial $G$-set. The profinite fundamental group of $\hat{X}_G$ is equal to the $G$-equivariant completion of the fundamental group of $X$, i.e. the canonical map $\hat{(\pi_1X)}_G \to \pi_1\hat{X}_G$ is an isomorphism of profinite $G$-groups.
\end{prop}
\begin{proof}
The fundamental group of a profinite space is defined as the limit of the finite automorphism groups of the finite Galois coverings of $X$. Each covering inherits an action by $G$ from $X$. Hence $G$ also acts continuously on the finite automorphism groups. Their limit is the profinite fundamental group of $\hat{X}_G$ by definition, but also the $G$-completion of the fundamental group $\pi_1X$ of the simplicial set $X$. 
\end{proof}

We equip the category of simplicial $G$-sets $\Sh_{G}$ with the model structure of \cite{gj} V \S 2. In this model structure a map $f:X\to Y$ is a weak equivalence (respectively fibration) in $\Sh_{G}$ if its underlying map in $\Sh$ is a weak equivalence (respectively fibration); and $f$ is a cofibration in $\Sh_{G}$ if $f$ is a monomorphism and $G$ acts freely on $Y_n - f(X_n)$ for each $ n\geq 0$.
There is the following partial analogue of Proposition \ref{adjcompletion}.
\begin{prop}\label{adjGcompletion}
Let $G$ be a strongly complete profinite group. The forgetful functor $|\cdot|:\hShg \to \Sh_{G}$ preserves fibrations and weak equivalences between fibrant objects.
\end{prop}
\begin{proof}
In both $G$-model structures the weak equivalences and fibrations are determined by the underlying non-equivariant maps. The assertion is then a consequence of the second part of Proposition \ref{adjcompletion}.
\end{proof}
\subsection{$G$-equivariant fibrant replacements}
Let $G$ be again an arbitrary profinite group. We want to define an explicit fibrant replacement functor in the model structure on $\hShg$ of Theorem \ref{Gmodelunstable}. After the discussion in the two previous sections one would expect that we had to significantly modify the construction for the fibrant replacement in $\hSh$ in order make it $G$-equivariant and continuous. But it turns out that the only necessary change is that we have to decompose $X$ with respect to its $G$-invariant  equivalence relations. The continuity comes for free (as we will see below, in the true sense of the word).\\ 
\begin{remark}\label{5.6.1}
We recall the following fact from \cite{ribes}, Remark 5.6.1. Let $G$ be a profinite group acting continuously on a profinite set $S$. It induces a homomorphism $G \to \mathrm{Homeo}(S)$ from $G$ to the group of homeomorphisms of $S$. When we equip $\mathrm{Homeo}(S)$ with the compact-open topology, this homomorphism is continuous if and only if the action of $G$ on $S$ is continuous by \cite{bourb}, X 3.4 Th\'eor\`eme 3. If $S$ is finite and discrete, the finite group $\mathrm{Homeo}(S)$ is just a discrete group.
\end{remark}
This implies the following crucial observation for free profinite groups. 
\begin{lemma}\label{5.6.2}
Let $G$ be a profinite group acting on a profinite set $S$. Then this action extends to a continuous action of $G$ on the free profinite group $\hat{F}(S)$ on $S$. 
\end{lemma}
\begin{proof}
The universal property of $\hat{F}(S)$ implies two things. First, the action of $G$ on $S$ induces an action on $\hat{F}(S)$ since the continuous map $\varphi_g:= \iota \circ g: S \to \hat{F}(S)$ induces a unique continuous group homomorphism $\bar{\varphi}_g:\hat{F}(S) \to \hat{F}(S)$ such that the diagram
$$\xymatrix{
S \ar[d]_{\iota} \ar[r]^g & S \ar[d]^{\iota} \\
\hat{F}(S) \ar[r]_{\bar{\varphi}_g} & \hat{F}(S)}$$
commutes. 
Second, there is an equality $\mathrm{Homeo}(S)=\Aut(\hat{F}(S))$ between the space $\mathrm{Homeo}(S)$ and the space of continuous automorphisms of $\hat{F}(S)$ with compact-open topology. By Remark \ref{5.6.1}, the action of $G$ on $S$ is continuous if and only if the homomorphism $G\to \mathrm{Homeo}(S)$ is continuous. By the same argument of \cite{bourb}, X 3.4 Th\'eor\`eme 3, the action of $G$ on $\Aut(\hat{F}(S))$ is continuous if and only if the induced homomorphism $G\to \Aut(\hat{F}(S))$ is continuous. This shows the assertion.
\end{proof}

We are now prepared for the fibrant replacement in $\hShg$. Recall that when $X$ is a reduced simplicial finite set, the first step in the construction of the fibrant replacement in $\hSh$ is to apply the free profinite loop group construction on $X$, which is the simplicial profinite group $\hat{\Gamma}X$ given in degree $n$ by the free profinite group on the finite set $X_{n+1} - s_0X_n$. 
Now Lemma \ref{5.6.2} shows that when $X$ is a simplicial finite discrete $G$-set, then $\hat{\Gamma}X$ is a simplicial profinite $G$-group. Hence the profinite classifying space $\bar{W}\hat{\Gamma}X$ is a profinite $G$-space and the canonical map $\heta:X \to \bar{W}\hat{\Gamma}X$ is a map in $\hShg$. Since a map is a weak equivalence (respectively fibration) in $\hShg$ if its underlying map in $\hSh$ is a weak equivalence (respectively fibration), $\heta$ is a functorial fibrant replacement in $\hShg$ by Theorem \ref{fibrep}.\\
Now let $X$ be an arbitrary simplicial finite set. In \S\,\ref{secfibrep}, we took the free profinite loop groupoid $\hat{\Gamma}X$ on $X$. It inherits a $G$-action from $X$, where an action on a groupoid can be described as follows. Let $\Gamma$ be a groupoid. A $G$-action on $\Gamma$ is a $G$-action $\mu_0$ on the set of objects of $\Gamma$ and an action $\mu_1$ of $G$ on the set of morphisms of $\Gamma$ such that $\mu_1(g, \gamma_1 \gamma_2)=\mu_1(g,\gamma_1) \circ \mu_1(g,\gamma_2)$ as morphisms of $\Gamma$ for every $g\in G$ and all morphisms $\gamma_1$ and $\gamma_2$ of $\Gamma$.\\ 
Since $\hat{\Gamma}X$ is the free profinite groupoid on $X$, the same argument as for the free profinite group on a set shows that the induced action of $G$ on $\hat{\Gamma}X$ is in fact continuous, i.e. $\mu_0$ and $\mu_1$ act continuously on the set of objects and morphisms of $\hat{\Gamma}X$ respectively. So the classifying space $\bar{W}\hat{\Gamma}X$ is again a profinite $G$-space whose underlying profinite space is fibrant.\\ 
%
Finally, for an arbitrary profinite $G$-space $X$, we use that it is isomorphic in $\hShg$ to the limit $\lim_{R_G\in \Rh_G(X)} X/R_G$ of simplicial finite $G$-sets by Lemma \ref{5.6.4}, where $\Rh_G(X)$ denotes the set of simplicial $G$-invariant open equivalence relations on $X$.  
We construct the free simplicial profinite groupoid on $X$ by 
$$\hat{\Gamma}X = \lim_{R_G \in \Rh_G(X)} \hat{\Gamma}(X/R_G).$$
By the previous argument, $G$ acts continuously on this free profinite groupoid. 
The application of the classifying space functor $\bar{W}$ yields a fibrant profinite $G$-space $\bar{W}\hat{\Gamma}X$ equipped with a canonical map 
$$\heta: X\to \bar{W}\hat{\Gamma}X=\lim_{R_G \in \Rh_G(X)} \bar{W}\hat{\Gamma}(X/R_G).$$
Since the construction of $\heta$ is a natural in $X$, this yields a functorial fibrant replacement in $\hShg$. Summarizing this discussion, we have proven the following result.
\begin{theorem}\label{Gfibrep}
Let $G$ be a profinite group and let $X$ be a profinite $G$-space. The map $\heta: X \to \bar{W}\hat{\Gamma}X$ is a weak equivalence in $\hShg$. The profinite $G$-space $\bar{W}\hat{\Gamma}X$ is fibrant in $\hSh$ and $\hShg$. Hence $\heta$ defines a functorial fibrant replacement in $\hShg$.
\end{theorem}
%
%
\begin{cor}
Let $G$ be a profinite group and $X\in \Sh_{G}$. There is a $G$-equivariant profinite completion functor that sends $X$ to a limit of finite spaces which are simplicial finite $G$-sets:
$$X\mapsto \hat{X}_G \mapsto \hat{X}_{f,G}:= \lim_{n,R_G,U} \cosk_n \bar{W}(\Gamma(X/R_G))/U \in \hShg.$$ 
\end{cor}

\begin{thebibliography}{99}
%
\bibitem{artinmazur} 
M. Artin, B. Mazur, 
Etale homotopy, 
Lecture Notes in Mathematics, 100, Springer, 1969.
\bibitem{boggi}
M. Boggi,
Profinite Teichm\"uller theory,
Math. Nachr. 279 (2006), 953-987.
\bibitem{boggi2}
M. Boggi,
Faithfulness of Galois representations associated to hyperbolic curves,
preprint, 2009, arXiv:0910.4305v2.
\bibitem{bouskan} 
A. K. Bousfield, D. M. Kan, 
Homotopy limits, Completions and Localizations, 
Lecture Notes in Mathematics, 304, Springer-Verlag, 1972.
\bibitem{bourb} 
N. Bourbaki, 
Topologie G\'en\'erale, 
Hermann, 1971.
\bibitem{dwyfried} 
W. G. Dwyer, E. M. Friedlander, 
Algebraic and Etale K-Theory,
Trans. Amer. Math. Soc. 292 (1985), 247-280.
\bibitem{dk2} 
W. G. Dwyer, D. M. Kan, 
An obstruction theory for diagrams of simplicial sets,
 Nederl. Akad. Wetensch. Indag. Math. 46 (1984), no. 2, 139Ð146.
\bibitem{dk} 
W. G. Dwyer, D. M. Kan, 
Homotopy theory and simplicial groupoids, 
Nederl. Akad. Wetensch. Indag. Math. 46 (1984), no. 4, 379-385.
\bibitem{fried} 
E. M. Friedlander, 
Etale homotopy of simplicial schemes,
Annals of Mathematical Studies, 104, Princeton University Press, 1982.
\bibitem{goerss} 
P. G. Goerss, 
Homotopy Fixed Points for Galois Groups,
in The Cech centennial (Boston, 1993), 
Contemp. Math. 181, 1995, pp. 187-224.
\bibitem{goersscomp} 
P. G. Goerss, 
Comparing completions of a space at a prime, 
In: Homotopy theory via algebraic geometry and group representations, 
Evanston, IL, 1997,  
Contemp. Math. 220, Amer. Math. Soc., Providence, RI, 1998, pp.65-102.
\bibitem{gj} P. G. Goerss, J. F. Jardine, 
Simplicial Homotopy Theory, 
Birkh\"auser Verlag, 1999.
\bibitem{calclim} 
D. C. Isaksen, 
Calculating limits and colimits in pro-categories,
Fund. Math. 175, 2002, 175-194.
\bibitem{compofpro} 
D. C. Isaksen, 
Completions of pro-spaces,
Math. Zeit. 250, 2005, 113-143.
\bibitem{jw} 
U. Jannsen, K. Wingberg, 
Die Struktur der absoluten Galoisgruppe $p$-adischer Zahlkšrper, 
Invent. Math. 70 (1982), 71-98. 
\bibitem{lochak}
P. Lochak,
Results and conjectures in profinite Teichm\"uller theory,
preprint, 2011, available at http://www.math.jussieu.fr/~lochak/.
\bibitem{may} 
J. P. May, 
Simplicial objects in algebraic topology,
Van Nostrand Math. Studies, 11, 1967.
\bibitem{ensprofin} 
F. Morel, 
Ensembles profinis simpliciaux etinterpr\'etation g\'eom\'etrique du foncteur T, 
Bull. Soc. Math. France 124 (1996), 347-373.
\bibitem{niksegal} 
N. Nikolov, D. Segal, 
On finitely generated profinite groups. I. Strong completeness and uniform bounds, 
Ann. of Math. 165 (2007), 171-238. 
\bibitem{profinhom} 
G. Quick, 
Profinite homotopy theory, 
Doc. Math. 13 (2008), 585-612.
\bibitem{gspaces} 
G. Quick, 
Continuous group actions on profinite spaces, 
J. Pure Appl. Algebra 215 (2011), 1024-1039.
\bibitem{homalg} 
D. G. Quillen, 
Homotopical algebra,
Lecture Notes in Mathematics, 43, Springer-Verlag, 1967.
\bibitem{quillen2} 
D. G. Quillen, 
An application of simplicial profinite groups, 
Comment. Math. Helv. 44 (1969), 45-60.
\bibitem{ravenel}
D. Ravenel,
The Structure of Morava Stabilizer Algebras,
Invent. Math. 37 (1976), 109-120.
\bibitem{rector} 
D. L. Rector, 
Homotopy theory of rigid profinite spaces I, 
Pacif. J. of Math. 85 (1979), 413-445.
\bibitem{ribes} 
L. Ribes, P. Zalesskii, 
Profinite Groups, 
Ergebnisse der Mathematik und ihrer Grenzgebiete, 40, Springer Verlag, 2000.
\bibitem{serre} 
J. P. Serre, 
Cohomologie Galoisienne, 
Lecture Notes in Mathematics, 5, Springer Verlag, 1965.
\bibitem{sullivan} 
D. Sullivan, 
Genetics of Homotopy Theory and the Adams Conjecture, 
Ann. of Math. 100 (1974), 1-79.
%
\end{thebibliography}
\end{document}